\documentclass[letterpaper,11pt]{article}

\usepackage{lipsum}
\usepackage{amsfonts}
\usepackage{graphicx}
\usepackage{epstopdf}
\usepackage{algorithmic}
\ifpdf
  \DeclareGraphicsExtensions{.eps,.pdf,.png,.jpg}
\else
  \DeclareGraphicsExtensions{.eps}
\fi

\title{Edge Deletion Algorithms for Minimizing Spread in SIR Epidemic Models
\thanks{The first author and the fourth author are supported by Knut \& Alice Wallenberg foundation, and by Swedish Research Council. The second author is supported in part by National Science Foundation under grants CCF-1955351 and HDR TRIPODS CCF-1934931. The third author is supported in part by the C3.ai Digital Transformation Institute sponsored by C3.ai Inc. and the Microsoft Corporation and in part by the National Science Foundation under Grants NFS-CNS-2028738 and NFS-ECCS-2032258. }
}

\author{Yuhao Yi\footnotemark[1]
\and Liren Shan\footnotemark[2]
\and Philip E. Par\'{e}\footnotemark[3]
\and Karl H. Johansson\footnotemark[1]
}

\date{}
  
\usepackage{amsopn}
\DeclareMathOperator{\diag}{diag}

 \usepackage[margin=1in]{geometry}
\usepackage[algo2e,ruled,vlined]{algorithm2e}
\usepackage{url}
\usepackage{amssymb}
\usepackage{amsthm}
\usepackage{mathtools}
\usepackage[caption=false]{subfig}

\newtheorem{theorem}{Theorem}[section]

\newtheorem{lemma}[theorem]{Lemma}

\newtheorem{proposition}[theorem]{Proposition}

\newtheorem{example}[theorem]{Example}
\newtheorem{definition}[theorem]{Definition}
\newtheorem{remark}[theorem]{Remark}
\newtheorem{problem}{Problem}

\newcommand{\1}{\mathbf{1}}
\newcommand{\0}{\mathbf{0}}

\newcommand{\Erdos}{Erd\H{o}s--}
\newcommand{\Renyi}{R\'{e}nyi }
\newcommand{\eps}{\epsilon}
\renewcommand{\emptyset}{\varnothing}

\def\diag#1{\mathrm{diag}\left(#1 \right)}
\def\sat#1{\mathrm{sat}\left(#1 \right)}

\def\defeq{\stackrel{\mathrm{def}}{=}}
\def\prob#1#2{\mathbb{P}_{#1}\left[ #2 \right]}

\def\expec#1#2{{\mathbb{E}}_{#1}\left[ #2 \right]}

\newcommand{\norm}[1]{\left\lVert#1\right\rVert}

\newcommand{\SBM}{\mathcal{SBM}}
\newcommand{\calG}{\mathcal{G}}
\newcommand{\calS}{\mathcal{S}}
\newcommand{\calC}{\mathcal{C}}

\renewcommand\aa{\boldsymbol{\mathit{a}}}
\newcommand\mm{\boldsymbol{\mathit{m}}}
\newcommand\rr{\boldsymbol{\mathit{r}}}
\renewcommand\ss{\boldsymbol{\mathit{s}}}
\newcommand\xx{\boldsymbol{\mathit{x}}}
\newcommand\ee{\boldsymbol{\mathit{e}}}

\newcommand\BB{\boldsymbol{\mathit{B}}}
\newcommand\DDelta{\boldsymbol{\mathit{D}}}
\newcommand\DD{\boldsymbol{\mathit{D}}}

\newcommand\II{\boldsymbol{\mathit{I}}}

\newcommand\MM{\boldsymbol{\mathit{M}}}

\newcommand\QQ{\boldsymbol{\mathit{Q}}}
\newcommand\RR{\boldsymbol{\mathit{R}}}

\newcommand\XX{\boldsymbol{\mathit{X}}}

\DeclareMathOperator*{\argmin}{arg\,min}
\DeclareMathOperator*{\argmax}{arg\,max}

\usepackage{breakurl}

\begin{document}

\maketitle

\renewcommand{\thefootnote}{\fnsymbol{footnote}}

\footnotetext[1]{Division
of Decision and Control Systems, School of Electrical Engineering and
Computer Science, KTH Royal Institute of Technology, Stockholm, Sweden.
(\texttt{yuhaoy@kth.se}, \texttt{kallej@kth.se}).}

\footnotetext[2]{Department of Computer Science, Northwestern University, Evanston, IL, USA. (\texttt{lirenshan2023@u.northwestern.edu}).}
  
\footnotetext[3]{School of Electrical and Computer Engineering, Purdue University, West Lafayette, IN, USA. (\texttt{philpare@purdue.edu}).}

\begin{abstract}
This paper studies algorithmic strategies to effectively reduce the number of infections in susceptible-infected-recovered (SIR) epidemic models. We consider a Markov chain SIR model and its two instantiations in the deterministic SIR (D-SIR) model and the independent cascade SIR (IC-SIR) model. We investigate the problem of minimizing the number of infections by restricting contacts under realistic constraints. Under moderate assumptions on the reproduction number, we prove that the infection numbers are bounded by supermodular functions in the D-SIR model and the IC-SIR model for large classes of random networks. We propose efficient algorithms with approximation guarantees to minimize infections. The theoretical results are illustrated by numerical simulations.
\end{abstract}

\section{Introduction}

Epidemic spreading processes can significantly disrupt the functioning of the society and pose risks to the health of individuals. Therefore, the study of epidemic spread has a long history.
One of the most fundamental models is the susceptible-infected-recovered (SIR) model, where each individual can be in one of three states: susceptible, infected, or recovered. Variations of the SIR model have been proposed, categorized as compartment-based models and networked models.  An overview of SIR models is provided in \cite{youssef2011individual}. The behaviors of the SIR spread dynamics are extensively studied~\cite{kermack1927contribution,abbey1952examination,draief2008thresholds,borgs2010distribute}. 

The problem of efficiently controlling epidemic spread has received increasing attention. One line of work has focused on minimizing the spectral norm of the transition matrix to suppress the process~\cite{wan2007network,wan2008designing,MSKL+11,bishop2011link,PZEJ+14,mai2018distributed}. Resource allocation~\cite{PZEJ+14} and network modifying~\cite{wan2007network, prakash2013fractional,pare2017epidemic} strategies have also been studied. Optimal control problems have been studied to minimize the cost over a given horizon~\cite{khanafer2014optimal}. Since the total number of infections is an important criterion, problems of minimizing infections have been proposed~\cite{prakash2013fractional,OP16}.

In this paper, we consider optimization problems that minimize infections in a networked SIR Markov chain model. We provide efficient strategies of modifying the structure of the contact network to minimize the number of infections. We show the effectiveness of the proposed algorithms through theoretical characterizations and numerical evaluations.

\subsection{Related Work} 

Followed from the Markov chain susceptible-infected-susceptible (SIS) model~\cite{van2009virus}, continues and discrete Markov chain SIR models have been proposed in~\cite{youssef2011individual} and~\cite{RH15}. Deterministic models based on mean-field approximations in the Markov chain models have also been studied in these papers. 
%
The independent cascade SIR (IC-SIR) model can be viewed as a networked extension to the Reed-Frost model~\cite{draief2008thresholds}, which is one of the earliest SIR models studied in depth~\cite{abbey1952examination}. The IC-SIR model was proposed in~\cite{GLM01}, popularized in~\cite{KKT03}, and has a rich volume of follow-up studies on efficient algorithms for influence maximization~\cite{BBCL14}, network design~\cite{WKSZ17}, and inferring network structure~\cite{NS12}.


Some researchers have
focused on problems of minimizing the spectral norm of a parameter matrix by removing nodes or edges. These problems have been proven to be $\mathbf{NP}$-complete and $\mathbf{NP}$-hard, respectively~\cite{MSKL+11}. 
Several heuristic algorithms are proposed based on betweenness centrality~\cite{HK02,SMHH11} or convex relaxation of the original problem~\cite{bishop2011link}. The continuous version of the problem have been discussed in~\cite{preciado2013optimal,PZEJ+14,mai2018distributed,pareautomatica,hota2020closed}.

As for minimizing number of infections, a graph partitioning problem has been investigated in \cite{enns2012optimal}, without considering any dynamics or random initial conditions in the network. A budget-constraint resource allocation problem has been studied in~\cite{OP16}. The problem is formulated as a continuous optimization by modifying the infection rates and recovery rates for nodes under the assumption that the system is stable. To the best of our knowledge, no algorithm with provable approximation guarantee has been proposed for the node and edge removal problems to minimize infections.

\subsection{Contributions}
We start by proposing a general Markov chain SIR (G-SIR) model. Two common SIR models are simplifications to the G-SIR model, the deterministic SIR (D-SIR) model and the IC-SIR model.

Our main results focus on the problem of minimizing the number of infections by removing edges in the contact network from a given candidate set. The candidate set models the reality that only limited types of contacts can be removed. With moderate restrictions on the reproduction number, we show that the problems can be efficiently solved by using greedy algorithms.

For the D-SIR model, we propose an upper bound for the number of infections as a surrogate objective function. We prove that when the system is exponentially stable, the surrogate function is a monotone supermodular function of the set of deleted
edges, although in a relaxed problem the same function is not convex with respect to the edge weights. Then it is shown that the number of infections can be reduced by efficiently minimizing the upper bound.

For the IC-SIR model, we study the expected number of infections. In this model, the epidemic spreads over a contagion network randomly sampled from a given contact network. We consider contact networks generated from \Erdos \Renyi (ER) graphs and the stochastic block model (SBM). 
The statistics of the contagion network encode 
many properties of the spread process. For example, the expected average degree of the contagion network can be interpreted as the reproduction number. When the average degree of the contagion network is less than one, and the number of initial infections is no larger than $O(n^{\frac{1}{3}-c})$ where $n$ is the population size and $c$ is any positive number, we design efficient mitigation
algorithms. 
For a contact network generated from an ER graph, we prove that with high probability, the expected number of infections is well approximated by a monotone supermodular function. We obtain similar results for SBMs, with the additional assumption that the number of blocks is no larger than $O(\log{n})$.
We prove that the problem is $\mathbf{NP}$-hard in a regime where the reproduction number is greater than one, indicating that in this case the problem is computationally hard.

Our theoretical results are supported by numerical simulations. We show the effectiveness of the proposed algorithms by running simulations on both synthesized random networks and 
a real contact network.
We compare the results given by the 
algorithms using the D-SIR model and 
the IC-SIR model and corresponding results for the
G-SIR model. We find that both algorithms effectively reduce the number of infections also in the G-SIR model under the proposed conditions.

\subsection{Outline} The remainder of the paper is organized as follows. In Section~\ref{sec:pre} we introduce notations and review some basic concepts. In Section~\ref{sec:models} we describe the considered models. In Section~\ref{sec:problem} we present problem formulation. In Section~\ref{sec:dsir} we present a supermodular upper bound for the infections in the D-SIR model when the system is stable. 
We provide conditions such that, with high probability, the number of infections approximates a supermodular function,
addressing the cases where,
in Section~\ref{sec:IC_R} and~\ref{sec:IC_SBM},
the contact network is sampled from an ER graph and a SBM, 
respectively. In Section~\ref{sec:hardness} we prove the $\mathbf{NP}$-hardness of the problem. Section~\ref{sec:numerical} provides some numerical experiments and simulations to show the effectiveness of the proposed algorithms in all the considered models, followed by the Conclusion.
\section{Preliminaries}
\label{sec:pre}

Let $G=(V,E)$ be a graph with vertex (node) set $V$ and edge set $E$. $G$ can be either directed or undirected unless noted specifically. For a digraph, an edge from node $j$ to node $i$ is denoted as an ordered pair $(j,i)$; for a undirected graph the edge is denoted as an unordered pair $\{j,i\}$. We define the size of a graph or connected component as the number of vertices (nodes) in it. 

Let $\calG(n,p)$ denote the ER graph with $n$ vertices where each edge occurs with probability $p$. We define the SBM $\SBM(n,\kappa,\QQ)$ as follows:
Let $\{V_1,V_2,\cdots,V_\kappa\}$ be $\kappa$ communities with $n$ vertices. For each pair of vertices $u\in V_i$ and $v \in V_j$, the edge $(u,v)$ is sampled with probability $Q_{ij}$ independently.

In this paper we use bold font for matrices and vectors. We denote $\ee_i, i\in[n]$ the $i$-th canonical basis of $\mathbb{R}^n$. We use $\1$ and $\0$ to denote all one and all zero vectors. Further, we recall the definition and several properties of M-matrices. An M-matrix is a square matrix whose off-diagonal entries are non-positive and eigenvalues have non-negative real parts. It has the following property.
\begin{lemma}[\cite{plemmons1977m}]\label{lem:m-matrix}
A non-singular M-matrix is inverse-positive, which means the entries of the inverse are all non-negative.
\end{lemma} 

We adopt the standard asymptotic notation. Given two functions $f$ and $g$ of the variable $n$, we denote $f=O(g)$
if there exists an $n_0$ and a constant $c>0$, such that for all $n\geq n_0$, $f\leq c\cdot g$. Further, we denote $f=o(g)$  
if $f/g$ tends to zero when $n$ tends to infinity. The term \emph{with high probability} is used when the probability is $1-o(1)$.

The following
definitions are
frequently used.

\begin{definition}[Plus and minus operation for graphs]
Given two graphs $G_1=(V_1,E_1)$, $G_2=(V_2,E_2)$, 
$G'=(V',E')\defeq G_1+G_2$ 
is 
a new graph with vertex set $V'= V_1 \cup V_2$ and edge set $E'= E_1\cup E_2$, and 
$G''=(V'',E'')\defeq G_1-G_2$ 
is
a graph with $V''= V_1 \cup V_2$ and $E''= E_1\backslash E_2$.
\end{definition}

\begin{definition}[Monotonicity]
A set function $f:2^{\Omega}\mapsto \mathbb{R}$ is 
monotonically 
non-increasing if $f(P_1)\geq f(P_2)$ holds for all $P_1\subseteq P_2\subseteq \Omega$.
\end{definition}

\begin{definition}[Supermodularity]
A set function $f:2^{\Omega}\mapsto \mathbb{R}$ is supermodular if $f(P_1)-f(P_1\cup \{e\}) \geq f(P_2)-f(P_2\cup \{e\})$ for all $P_1\subseteq P_2\subseteq E$ and all $e\in (E\backslash P_2)$.
\end{definition}
A function $f(P)$ is submodular if $g(P)\defeq -f(P)$
is supermodular. A function $f(P)$ is modular if it is both submodular and supermodular.

\begin{definition}[$(1\pm\eps)$-approximation]
Given two real numbers $a,b \geq 0$ and an approximation error $\eps\geq 0$, if $a$ and $b$ satisfy
\begin{align*}
    (1-\eps)b \leq a\leq (1+\eps)b\,,
\end{align*}
then $b$ is an $(1\pm\eps)$-approximation of $a$, denoted $a\approx_{\eps}b$.
\end{definition}
We note that $a\approx_{o(1)}b$ if and only if $b\approx_{o(1)}a$.

\section{Model Description}
\label{sec:models}
In networked SIR models, each node, at a given time, is exclusively in one of three states: \emph{susceptible} ($S$), \emph{infected} ($I$), or \emph{removed} ($R$). For node $i$, we define the $\{0,1\}$ indicator variables for the states at time step $t$ as $S_i(t)$, $I_i(t)$, and $R_{i}(t)$. The variables take value $1$ if the node is in that state, and otherwise take value $0$. We note that $S_i(t)+I_i(t)+R_i(t)=1$. Motivated by~\cite{HH13,RH15}, we propose the G-SIR
model using a $3^n$-state Markov chain. The recursive relations between the states are defined as
\begin{align}
\label{eqn:markov_I}
    I_i(t+1) &= I_i(t)(1-\delta_i(t))+S_i(t)\left(1-\prod_{j=1}^n \left(1-\beta_{ij}(t)I_j(t)\right)\right)\,,\\
    R_i(t+1)&=R_i(t)+\delta_{i}(t)I_i(t)\,,
\end{align}
where $\beta_{ij}(t)$ is the random indicator variable for the event that node $i$ is infected from its infected neighbor $j$; $\delta_i(t)$ is the random indicator variable for recovering. We assume that $\beta_{ij}(t)$ for any $i,j$ are mutually independent. In addition, we assume that for a fixed $i,j$ pair, $\beta_{ij}(t)$ are $i.i.d$ with respect to $k$, and for a fixed $i$, $\delta_i(t)$ are $i.i.d.$ with respect to $t$.

We simplify the G-SIR model by using mean-field approximation in a similar manner as~\cite{van2014exact,cator2014nodal}. We take expectation for both sides of (\ref{eqn:markov_I}) and arrive at
\begin{align*}
    \expec{}{I_i(t+1)} &= \expec{}{I_i(t)}\expec{}{1-\delta_i(t)}+\expec{}{S_i(t)}-\expec{}{S_i(t)}\prod_{j=1}^n \expec{}{1-\beta_{ij}(t)I_j(t)}\,,
\end{align*}
by assuming that the states of all nodes are mutually independent for any given $t$. 
Let $x_i(t)\in [0,1]$ and $r_i(t)\in [0,1]$ be the probabilities of node $i$ being infected and being removed at time step $t$,
respectively.
Then, by using the approximation $1-x\approx e^{-x}$ for small $x$, the mean-field approximation of the corresponding discrete-time SIR model 
can be written as
\begin{align*}
    x_i(t+1) &= x_i(t) + \left(1 - x_i(t) - r_i(t)\right) 
    \sum_{j=1}^n B_{ij} x_j(t) 
    - D_{i} x_i(t),  \\
    r_i(t+1) &= r_i(t) + D_{i} x_i(t), 
\end{align*}
where $B_{ij}\defeq \expec{}{\beta_{ij}(t)}$ and $D_{i}\defeq \expec{}{\delta_{i}(t)}$.
$B_{ij}$ is the infection rate for edge $(j,i)$ and $D_i$ is the healing rate for node $i$. In some variants of this model, $B_{ij}$ is interpreted as the product of the infection rate and the edge weight~\cite{HH13,hota2020closed}.

We define the vectors $\xx(t),\rr(t)$ with entries $x_i(t)$ and $r_i(t)$ for all nodes $i$. Then the discrete-time SIR model can be expressed in matrix form as follows
\begin{align}
    \xx(t+1) &= \xx(t) +  (\II - \XX(t) - \RR(t))
    \BB\xx(t)
    -\DDelta  \xx(t), \label{eq:pdiscrete} \\
    \rr(t+1) &= \rr(t) +  \DDelta \xx(t), \label{eq:rdiscrete}
\end{align}
where the entries of $\BB$ are $B_{ij}$; 
$\DD$ is a diagonal matrix with diagonal entry $D_{i}$; 
and  
$\XX(t)$ and $\RR(t)$ are the diagonal 
matrices
$\diag{\xx(t)}$ and $\diag{\rr(t)}$, respectively. 

In this paper we study the cumulative number of infections,
defined as
$\mm(t)\defeq \xx(t)+\rr(t)$. We denote $\mm^*$ with entries $m^*_i\defeq \sup_{t}(m_i(t))$. Then $\norm{\mm^*-\mm(0)}_1$ can be interpreted as the
number of increased infections in the network. This vector $\mm(t)$ is known to be inside $[0,1]^n$ under the assumption that $D_i < 1$, $\sum_{j=1}^n B_{ij} <1$ for all $i\in[n]$
~\cite{hota2020closed}. Therefore we can treat $m_i(t)$ as the probability that node $i$ is in $I$ or $R$ at time $t$.
We refer to the model given by (\ref{eq:pdiscrete}) and (\ref{eq:rdiscrete}) as the D-SIR Model.
It is well known that the mean-field approximation gives an upper bound for the probabilities of infections in the G-SIR model~\cite{cator2014nodal}. 


We consider the IC-SIR Model~\cite{GLM01,KKT03,NS12} as another simplification to the 
G-SIR model, where $\delta_{i}(t)=1$ for all $i$ and $t$. 
In this model, we use random indicator variable vectors $\tilde{\xx}(t) \in \{0,1\}^n$ and $\tilde{\rr}(t) \in \{0,1\}^n$ to denote 
whether the nodes are
in states $I$ and $R$, respectively. The nonzero entries in $\tilde{\xx}(0)$ are the initial infected nodes, which are referred to as seeds. Let $s \defeq \norm{\tilde{\xx}(0)}_1$ be the number of initial infections.
We also define a matrix $\widetilde{\BB}(t)$, where the entries are random variables $\beta_{ij}(t)$ which are $i.i.d.$ with respect to $t$.
The success probability for each entry $\beta_{ij}(t)$ in $\widetilde{\BB}(t)$ is $p_{ij} \defeq \expec{}{\beta_{ij}(t)}$.
Thus, the IC-SIR model can be expressed as follows
\begin{align}
\tilde{\xx}(t+1) &= \sat{\widetilde{\BB}(t)\tilde{\xx}(t)+\tilde{\xx}(t)+\tilde{\rr}(t)}  - \tilde{\xx}(t) - \tilde{\rr}(t),\\
\tilde{\rr}(t+1) &= \tilde{\rr}(t) + \tilde{\xx}(t),
\end{align}
where the saturation function $\sat{\cdot}: \mathbb{R}^n \mapsto \mathbb{R}^n$ is defined as $\sat{\aa}_i=1$ if $a_i\geq 1$, and $\sat{\aa}_i=a_i$ if $a_i\leq 1$.

For each edge in the network, it can induce a new infection only when one of its incident nodes is infected and the other one is susceptible. Since any infected node recovers in one time step, 
each edge can change the number of infections at the first time that
any of its incident nodes are
infected.
Therefore, the random transition matrix $\widetilde{\BB}(t)$ can be sampled once beforehand, denoted as $\widetilde{\BB}$. This matrix $\widetilde{\BB}$ can be viewed as a $\{0,1\}$ adjacency matrix of a sampled network where the actual spreading happens. This network is called the contagion network. At the end of a spread process, nodes are either in the removed state or the susceptible state. Further, we use $\tilde{\rr}^*$ to denote the nodes eventually in state $R$: $\tilde{r}_i^*\defeq \sup_t(\tilde{r}_i(t))$ for each node $i$. 
Then $\norm{\tilde{\rr}^*-\tilde{\xx}(0)-\tilde{\rr}(0)}_1$ is the number of increased infections.


\section{Problem Formulation}
\label{sec:problem}
We consider the problems of minimizing the number of infections for the D-SIR model and the IC-SIR model.
\begin{problem}\label{prob:dm}
Given a digraph $G=(V,E)$, the infection rate $B_{ij}$ for each edge $(j,i)$ in $G$ and healing rate $D_{i}$ for every node $i\in V$, an initial state vector $\xx(0)\in [0,1]^{n}$ and $\rr(0)$ such that $\xx(0)+\rr(0) \in [0,1]^{n}$, a candidate edge deletion set $Q\subseteq E$, $|Q|=q$, an integer $0<k\leq q$, find a set of edges $P^*\subseteq Q$, where $|P^*|\leq k$, such that
\begin{align}
\label{obj:D_SIR}
    P^{*} \in \argmin_{P\subseteq Q,|P|=k} \sigma(P)\,,
\end{align}
where $\sigma(P)\defeq\norm{\mm^*-\mm(0)}_1$.
\end{problem}

\begin{problem}\label{prob:ic}
Given an undirected graph $G=(V,E)$, 
the 
activation
probability of each edge $p_{ij}\defeq\prob{}{ \widetilde{B}_{ij}(t)=1}$, an initial state vector $\tilde{\xx}(0)\in\{0,1\}^n$ with $s$ nonzero entries, 
a candidate edge deletion set $Q\subseteq
E$, where $|Q|=q$, 
and
an integer $0<k\leq q$, find a set of edges $P^*\subseteq Q$, where $|P^*|\leq k$, such that
\begin{align}\label{obj:IC-SIR}
    P^{*} \in \argmin_{P\subseteq Q,|P|=k} \expec{}{\tilde{\sigma}(P)}\,,
\end{align}
where $\tilde{\sigma}(P) \defeq \norm{\tilde{\rr}^*-\tilde{\xx}(0)}_1$. 
\end{problem}
We note that $\sigma(P)$ and $\tilde{\sigma}(P)$ are defined to measure the number of new infections. 

\begin{remark}
An extension of Problem~\ref{prob:ic} is to consider $\tilde{\xx}(0)\in\{0,1\}^n$, where each $\tilde{x}_i(0)$ is a Bernoulli random variable with mean $\mu_i$. Let $s = \sum_{i=1}^n \mu_i$ be the expected number of initial infections. By the Chernoff bound, the number of initial infections in $\tilde{\xx}(0)$ is at most $s+3\sqrt{s\log n}$ with high probability. 
\end{remark}



In this paper we find conditions for the D-SIR model and the IC-SIR model such that the objectives of Problem~\ref{prob:dm} and~\ref{prob:ic} are bounded by monotone
supermodular functions of the edge deletion set. 
For Problem~\ref{prob:dm}, when the dynamics of the D-SIR model are
exponentially stable, we find a monotone 
supermodular upper bound $\hat{\sigma}$ of the objective function. Therefore we use the upper bound as a surrogate to minimize. According to a well-known result of supermodular minimization~\cite{Nem78}, we obtain a greedy algorithm which gives a $(1- 1/e)$ approximation for the optimum decrease of the surrogate.

For Problem~\ref{prob:ic}, when the contact network is sampled from an ER graph or a SBM, the expected number of infections $\expec{}{\tilde{\sigma}}$ is a $(1\pm o(1))$-approximation to a monotone
supermodular function $\expec{}{\tilde{\sigma}'}$ with high probability. In this case we optimize the original objective function $\expec{}{\tilde{\sigma}}$ and treat the differences as errors. A greedy algorithm gives a $(1-1/e-o(1)-\eps)$-approximation~\cite{KKT03} for the optimal decrease in
the expected number of infections. 
We describe the greedy algorithm with a given set function $f(P)$ in Algorithm~\ref{alg:greedy}. We note that $f$ is defined as $f= \hat{\sigma}$ for Problem~\ref{prob:dm} and $f= \expec{}{\tilde{\sigma}'}$ for Problem~\ref{prob:ic}.

\begin{algorithm2e}[t]
\SetAlgoLined
\caption{Greedy algorithm}
\label{alg:greedy}

\SetKwData{Left}{left}\SetKwData{This}{this}\SetKwData{Up}{up}
\SetKwFunction{Union}{Union}\SetKwFunction{FindCompress}{FindCompress}
\SetKwInOut{Input}{Input}\SetKwInOut{Output}{Output}

\Input{a function $f\in \{\hat{\sigma}, \expec{}{\tilde{\sigma}'}\}$, a contact network $G$, initial states of nodes, a candidate edge set $Q$, an integer $k$;}
\Output{a set of $k$ edges $P \subseteq Q$;}

Initialize the set $P = \varnothing$;\\
\For{round $i$ from $1$ to $k$}{
Compute the corresponding function $f(P\cup\{e\})$ for any $e \in Q \setminus P$;\\
Set $e^* = \argmax_{e \in Q \setminus P} f(P) - f(P\cup\{e\})$;\\
Update $P = P \cup \{e^*\}$;
}
Return $P$;
\end{algorithm2e}

\section{Supermodularity for the D-SIR Model}
\label{sec:dsir}
In this section, we consider the problem of deleting edges from the current network to mitigate epidemic spread in the D-SIR model. First, we show that the cumulative number of infections is upper bounded by a supermodular function of all the edges in the graph.

Let $\MM_{-P} \defeq \II-\DDelta + (\II - \XX(0) - \RR(0))\BB_{-P}$ denote the transition matrix between $\xx(1)$ and $\xx(0)$, where $\BB_{-P}$ is the matrix obtained from $\BB$ by setting entries $B_{ij}=0$ for all edges $(j,i)\in P$. 
\begin{theorem}\label{thm:dm_upper}
Suppose the D-SIR model satisfies $\norm{\MM_{-P}} < 1$ for all $P \subseteq Q$, where $Q$ is the candidate edge deletion set. The number of infections is then at most 
\begin{equation}\label{eq:dm_upper_bound}
\sigma(P) \leq \hat{\sigma}(P) \defeq \1^\top (\MM_{-P}+\DDelta-\II)(\II-\MM_{-P})^{-1} \xx(0),
\end{equation}
where $\hat{\sigma}(P)$ is a monotone supermodular function of the edge deletion
set $P$.
\end{theorem}
Theorem~\ref{thm:dm_upper} shows that the objective function $\sigma(P)$ can be reduced efficiently by minimizing the supermodular function $\hat{\sigma}(P)$. 
We employ this fact to design a greedy algorithm, given in Algorithm~\ref{alg:greedy}.
The algorithm 
takes as input the function $\hat{\sigma}(P)$, the contact graph $G$, initial states of nodes, a candidate set $Q$, and an integer $k$. The algorithm  returns a set $P$ such that $\hat{\sigma}(\emptyset) - \hat{\sigma}(P)\geq (1-1/e)\left(\hat{\sigma}(\emptyset) - \hat{\sigma}(P^*)\right)\,,$ where $P^*$ is the optimal solution for $\sigma(P)$.

We show that the upper bound function $\hat{\sigma}(P)$ is monotone supermodular.

\begin{lemma}\label{lem:dsir_supermodular}
The function $\hat{\sigma}(P)$ is monotone supermodular with respect to the 
edge deletion set $P\subseteq Q$ when $\norm{\MM_{-P}}< 1$ for all set $P\subseteq Q$.
\end{lemma}

The proof can be found in Appendix~\ref{apx:dsir}.
Then, we prove Theorem~\ref{thm:dm_upper}.

\begin{proof}[Proof of Theorem~\ref{thm:dm_upper}]
By Equation (\ref{eq:pdiscrete}) in the D-SIR model, the infected state for any time step $t$ is
\begin{align*}
\xx(t+1) =& \xx(t) -  (\DDelta + (\XX(0) + \RR(0)) \BB_{-P}  - \BB_{-P}) \xx(t) \\
&- (\XX(t)+\RR(t) - (\XX(0) + \RR(0)))(\BB_{-P}) \xx(t)  \\
\leq&  \xx(t) - (\DDelta + (\XX(0) + \RR(0))\BB_{-P} - \BB_{-P}  ) \xx(t) \\
=& (\II-\DDelta + (\II -\XX(0) - \RR(0) ) \BB_{-P}) \xx(t)\,,
\end{align*}
where the inequality holds because the entries of $(\XX+\RR)$ are 
monotonically
non-decreasing
as a function of time $t$.
By applying this inequality recursively, we have for any time step $t$
\begin{align}\label{eqn:iterative}
\xx(t) \leq \left(\MM_{-P}\right)^{t} \xx(0)\,.
\end{align}
Thus, the cumulative number of infections at any time step $t$ is
\begin{align*}
\norm{\mm(t)-\mm(0)}_1 
&\leq \1^\top (\II-\XX(0)-\RR(0))(\BB_{-P}) \sum_{t' = 0}^{t-1} \xx(t')\\
&\leq \1^\top    (\MM_{-P}+\DDelta-\II)(\II-\MM_{-P})^{-1} \xx(0)\,,
\end{align*}
where the last inequality is due to the geometric series of matrices. 
\end{proof}

Now we provide
a sufficient condition for $\norm{\MM_{-P}}<1, \forall P\subseteq Q$, which also ensures 
exponential stability of the 
healthy state
starting from the initial state $\xx(0)$ and $\rr(0)$. Similar condition for the SIS model has been given in~\cite[Theorem 4]{GPSJ20}.
\begin{theorem}[Sufficient condition for exponentially stable]
\label{sufficient:theorem}
If there exists a positive number $\epsilon$ such that for any $i\in [n]$, $(1-x_i(0)-r_i(0)) \sum_{j=1}^n  B_{ij} \leq D_{i} - \epsilon$, then the spectral norm $\norm{\MM_{-P}} \leq 1-\epsilon$. Moreover, under this condition, $\hat{\xx}=0$ is the unique equilibrium and the system is exponentially stable.
\end{theorem}

\begin{proof}
For simplicity of notation, we use $\MM$ to denote the matrix $\MM_{-P}$ for any 
edge  
deletion 
set $P$. For any $i \in [n]$, we have 
that 
the entries of $\MM$ are
\begin{align*}
M_{ii} = 1 - D_{i}, \qquad \sum_{j\not= i} \lvert M_{ij}\rvert \leq  (1-x_i(0)-r_i(0))  \sum_{j=1}^n  B_{ij}.
\end{align*}
By Gershgorin circle theorem, the spectral norm of $\MM$ is bounded by $1-\epsilon$.
Thus, this transition matrix is a contraction mapping. Since we have $\xx(t+1) \leq \MM \xx(t)$ and $\xx(t)\geq\0$ for any $t$, we obtain $\norm{\xx(t+1)} \leq (1-\eps) \norm{\xx(t)}$ for any $t$. Therefore, by \eqref{eqn:iterative}, $\norm{\xx(t)} \leq (1-\eps)^{t}\norm{\xx(0)}\leq e^{-t\eps}\norm{\xx(0)}$, which means the state $\xx$ converges to the origin with a rate of at least $\eps$. Then we conclude that $\hat{\xx} = 0$ is the unique equilibrium and the considered system is exponentially stable.
\end{proof}

\begin{remark}
Alternative conditions can be given for Theorem~\ref{thm:dm_upper}. 
Given that $\norm{\MM_{-\emptyset}} < 1$ and 
$\MM_{-\emptyset}$ is irreducible, 
the system is exponentially stable for any 
deletion 
edge set $P \subseteq Q$, followed from $\norm{\MM_{-P}}<1$. 
The reasoning is given as follows: 
According to the Perron-Frobenius Theorem~\cite{maccluer2000many}, the spectral norm of $\MM_{-\emptyset}$ is its Perron root $\lambda_{\max}$. Due to the Wielandt's Theorem~\cite{maccluer2000many}, the spectral norm, upon deleting  edges, is less than
or equal to $\lambda_{\max}$. 
\end{remark}

Convex relaxation is widely used as an approach to design heuristic algorithms or approximation algorithms by using a proper rounding. In addition, the supermodularity of a problem sometimes coincides with the convexity of a relaxed problem; for example,  see~\cite{MA19}. However, we show that a direct relaxation of $\hat{\sigma}$ is not a convex function with respect to the edge weights.
\begin{example}
We consider two undirected (bidirectional) graphs $G_1$ and $G_2$ which
have two copies of the same vertex set $V=\{1,2,3\}$. $G_1$ has edges $\{\{1,2\},\{1,3\}\}$ and $G_2$ has edges $\{\{1,2\},\{2,3\}\}$. Let the infection rate $B_{ij}=1/12$ for each edge $(j,i)$ and the healing rate $D_{i}=1/4$ for each node $i$. Let $\BB_{(1)}$ and $\BB_{(2)}$ be the transition matrices of $G_1$ and $G_2$, respectively. We let $\xx(0)=[1,0,0]^\top$, $\rr(0)=[0,0,0]^\top$. Then we have $\MM_{(i)}= \II-\DD+(\II-\XX(0)-\RR(0))\BB_{(i)}$ for $i\in\{1,2\}$, and define $\hat{\sigma}(\MM)= \1^\top (\MM+\DDelta-\II)(\II-\MM)^{-1} \xx(0)$. Thus, we obtain $\hat{\sigma}(\MM_{(1)})=2/3$, $\hat{\sigma}(\MM_{(2)})=1/2$, and $\hat{\sigma}(\frac{\MM_{(1)}+\MM_{(2)}}{2})=3/5$. Thus, $\frac{\hat{\sigma}(\MM_{(1)})+\hat{\sigma}(\MM_{(2)})}{2}<\hat{\sigma}(\frac{\MM_{(1)}+\MM_{(2)}}{2})$ yields the non-convexity result.
\end{example}

\section{Supermodularity for the IC-SIR Model in Random Graphs}\label{sec:IC_R}

In the IC-SIR model, the spreading process only uses the activated edges. The activated edges are sampled from the contact network $G$, which is given as the input of any algorithm. The contagion network $\widetilde{G}$, composed of the activated edges, is randomly generated from $G$. We assume the set of seeds is first fixed before the contact network and the contagion network are generated. We also assume the initial removed state $\rr(0)=\0$ without 
loss
of generality.

In Subsection~\ref{random.sec} we consider the case where the contact network is a complete graph. Then, the contagion network $\widetilde{G}$ is 
sampled
from an ER random graph $\calG(n,p)$, where $p$ is the connection probability.
In Subsection~\ref{sec:ic_random_contact} we consider the case where the contact network $G$ is generated from an ER random graph $\calG(n,p_1)$,
and
the contagion network $\widetilde{G}$ is sampled from $G$ with probability $p$ for each edge in~$G$. 

\subsection{Expected Number of Infections for a Complete Contact Network}
\label{random.sec}

Next we consider the expected number of infections $\expec{}{\tilde{\sigma}(P)}$, which is the objective in Problem~\ref{prob:ic}. The expectation is taken over the contagion network sampled from $\calG(n,p)$, where the probability $p = d/n$ for some constant $d < 1$. 
We prove that $\expec{}{\tilde{\sigma}(P)}$ is a $(1\pm o(1))$-approximation of a supermodular function.

We use $\tau_1$ to denote the event that the contagion network includes connected components with size greater than $L = 9(1-d)^{-2}\ln{n}$. Let $\tau_2$ be the event that there is at least one component with more than one seed. The complements of these events are denoted as $\bar{\tau}_1$ and $\bar{\tau}_2$, respectively. Let $\rho(P)$ denote the number of infected vertices in connect components that are trees with at least one seed. 

\begin{theorem}\label{thm:ic_approx}
If the number of seeds $s=O(n^{\frac{1}{3}-c})$ for any constant $c>0$, then for all $P\subseteq Q$,
\begin{align*}
\expec{\widetilde{G}}{\tilde{\sigma}(P)} &\approx_{o(1)} \expec{\widetilde{G}}{\tilde{\sigma}'(P)}\,, 
\end{align*}
where $\expec{\widetilde{G}}{\tilde{\sigma}'(P)}\defeq\expec{\widetilde{G}}{\rho(P) \mid \bar{\tau}_1,\bar{\tau}_2}$ is a monotone supermodular function. 
\end{theorem}

Theorem~\ref{thm:ic_approx} shows that we can directly optimize the objective function as an approximation to a monotone supermodular function.
We employ this fact to design a greedy algorithm, given in Algorithm~\ref{alg:greedy}.
In particular, Algorithm~\ref{alg:greedy} takes as inputs the function $\expec{\widetilde{G}}{\tilde{\sigma}'(P)}$, initial states of nodes, a candidate set $Q$, and an integer $k$. Assuming  $\expec{\widetilde{G}}{\tilde{\sigma}(\emptyset)}-\expec{\widetilde{G}}{\tilde{\sigma}(P^*)}$ is at least $1$, the algorithm returns a set $P$ such that 
\begin{align*}
\expec{\widetilde{G}}{\tilde{\sigma}(\emptyset)} - \expec{\widetilde{G}}{\tilde{\sigma}(P)} \geq (1-1/e-o(1)-\eps)\left(\expec{\widetilde{G}}{\tilde{\sigma}(\emptyset)}-\expec{\widetilde{G}}{\tilde{\sigma}(P^*)}\right),
\end{align*} 
where $P^*$ is the optimum solution for $\expec{\widetilde{G}}{\tilde{\sigma}(P)}$, and the error term $\eps$ is from a $(1\pm \eps)$-approximation of $\expec{\widetilde{G}}{\tilde{\sigma}'(P)}$ and its marginal gains. 
We will discuss the efficient computation of $\expec{\widetilde{G}}{\tilde{\sigma}'(P)}$ in  Section~\ref{dis:effiency}.

For the sake of analysis, we consider a modified process, which 
first generates
the random graph and then chooses
seeds randomly. 
As long as the information of the graph $\widetilde{G}$ is not revealed, 
this process
is equivalent to the one
where seeds are first chosen and 
then the random graph is generated.


To prove Theorem~\ref{thm:ic_approx}, we prepare Lemmas~\ref{lem:cc_size},~\ref{lem:seed_collision}, and~\ref{lem:cycle}. 
The proofs of these lemmas 
can be found in
Appendix~\ref{apx:icsir_random}.

First, we show that the event $\tau_1$ happens with a small probability. 
\begin{lemma}\label{lem:cc_size}
The probability that the random graph (with $p<1/n$) contains a connected component with size greater than $L= 9(1-d)^{-2}\ln{n}$ is at most
\begin{align*}
\prob{}{\tau_1} \leq n^{-2}.
\end{align*}
\end{lemma}



We also want to bound the probability of the seeds collision, i.e.\ the event that two seeds are sampled from the same component. This can be seen as a $\textit{balls into bins}$ problem. 

\begin{lemma}\label{lem:seed_collision}
The probability that there is at least one component with more than one seed conditioned on event $\bar{\tau}_1$ is
\begin{align*}
\prob{}{\tau_2 \mid \bar{\tau}_1} \leq \frac{2s^2L}{n}, 
\end{align*}
where $s$ is the number of seeds as defined in Problem~\ref{prob:ic}.
\end{lemma}

 
 
We use $y(P) = \tilde{\sigma}(P) - \rho(P)$ to denote the number of vertices in 
connect components that contain
at least one seed and one cycle. The following lemma follows in a similar manner to~\cite[Theorem 8.9]{BHK20}. 

\begin{lemma}\label{lem:cycle}
Given that the events $\bar{\tau}_1,\bar{\tau}_2$ happen, the expected number of vertices in 
connect components that contain
at least one seed and one cycle satisfies
\begin{align*}
    \expec{}{y(P) \mid \bar{\tau}_1,\bar{\tau}_2} \leq \frac{sL^2d^3}{2n(1-d)}\cdot\frac{1}{\prob{}{\bar{\tau}_1,\bar{\tau}_2}}\,.
\end{align*}
\end{lemma}

Now we prove the main theorem.
\begin{proof}[Proof of Theorem~\ref{thm:ic_approx}]

First, we show that $\expec{\widetilde{G}}{\rho(P) \mid \bar{\tau}_1,\bar{\tau}_2}$ is a good approximation of the expected number of infections $\expec{\widetilde{G}}{\tilde{\sigma}(P)}$.

By using the law of total expectation, the expected number of infections $\tilde{\sigma}(P)$ can be divided into three parts as follows,
\begin{align*}
\expec{\widetilde{G}}{\tilde{\sigma}(P)} = \expec{\widetilde{G}}{\tilde{\sigma} \mid \tau_1} \prob{}{\tau_1} + \expec{\widetilde{G}}{\tilde{\sigma} \mid \bar{\tau}_1,\tau_2} \prob{}{\bar{\tau}_1,\tau_2} + \expec{\widetilde{G}}{\tilde{\sigma} \mid \bar{\tau}_1,\bar{\tau}_2} \prob{}{\bar{\tau}_1,\bar{\tau}_2}.
\end{align*}
Then, we give upper bounds of these three terms individually. Note that the first term corresponds to the event that large components exist.
By Lemma~\ref{lem:cc_size}, we have
\begin{align*}
\expec{\widetilde{G}}{\tilde{\sigma}(P) \mid \tau_1} \prob{}{\tau_1} \leq n\cdot\frac{1}{n^2} = \frac{1}{n}.  
\end{align*}
When the largest component in $\widetilde{G}$ has size at most $L$, the number of infections is at most~$sL$. By Lemma~\ref{lem:seed_collision}, the second term for the seeds collision case is 
\begin{align*}
\expec{\widetilde{G}}{\tilde{\sigma}(P) \mid \bar{\tau}_1,\tau_2} \prob{}{\bar{\tau}_1,\tau_2} \leq sL\cdot \prob{}{\bar{\tau}_1,\tau_2} \leq \frac{2s^3L^2}{n}.
\end{align*}
For the third term, we separate the infections in connected components with cycles from the infections in those without cycles, $\tilde{\sigma}(P) = \rho(P)+y(P)$. By Lemma~\ref{lem:cycle}, we derive
\begin{align*}
\expec{\widetilde{G}}{\tilde{\sigma}(P) \mid \bar{\tau}_1,\bar{\tau}_2} \prob{}{\bar{\tau}_1,\bar{\tau}_2} \leq \expec{\widetilde{G}}{\rho(P) \mid \bar{\tau}_1,\bar{\tau}_2} \prob{}{\bar{\tau}_1,\bar{\tau}_2} + \frac{sL^2d^3}{2n(1-d)}.
\end{align*}
Thus, the expected number of infections $\tilde{\sigma}(P)$ is at most
\begin{align*}
\expec{\widetilde{G}}{\tilde{\sigma}(P)} \leq \expec{\widetilde{G}}{\rho(P) \mid \bar{\tau}_1,\bar{\tau}_2} + \frac{sL^2d^3}{2n(1-d)} + \frac{2s^3L^2}{n} + \frac{1}{n}.
\end{align*}
By Lemma~\ref{lem:cc_size} and Lemma~\ref{lem:seed_collision}, the expected number of infections can be lower bounded by
\begin{align*}
\expec{\widetilde{G}}{\tilde{\sigma}(P)} \geq \expec{}{\tilde{\sigma}(P) \mid \bar{\tau}_1,\bar{\tau}_2} \prob{}{\bar{\tau}_1,\bar{\tau}_2} \geq \left(1-\frac{1}{n^2}\right)\left(1-\frac{2s^2L}{n}\right)\expec{}{\rho(P) \mid \bar{\tau}_1,\bar{\tau}_2}.
\end{align*}
If the number of seeds $s=O(n^{\frac{1}{3}-c})$ for any constant $c>0$, then we have
\begin{align*}
 \expec{\widetilde{G}}{\tilde{\sigma}(P)} \approx_{o(1)} \expec{\widetilde{G}}{\rho(P) \mid \bar{\tau}_1,\bar{\tau}_2}.
\end{align*}

Given the events $\bar{\tau}_1,\bar{\tau}_2$ happen, the connected components considered in $\rho(P)$ are all trees with exactly one seed. For such connected components, it is easy to check that the number of infected vertices $\rho(P)$ is a monotone supermodular function with respect to the deletion edges. Since the momotonicity and supermodularity of $\rho(P)$ holds for all instances conditioned on $\bar{\tau}_1$ and $\bar{\tau}_2$, it also holds for the expectation $\expec{}{\rho(P)\mid \bar{\tau}_1,\bar{\tau}_2}$.
%
%
\end{proof}

\subsection{Expected Number of Infections for a Random Contact Network}
\label{sec:ic_random_contact}

In this section, we assume the contact network $G$ is generated from a random graph $\calG(n,p_1)$, and the contagion network $\widetilde{G}$ is again sampled from $G$ with probability $p$ for each edge. 
Further, the set of seeds is fixed before the contact network is generated. Combining these two random processes, the contagion network $\widetilde{G}$ can be seen as 
being generated from a random graph $\calG(n,p_1p)$. If the probability $p_1p < 1$, we have already shown that $\expec{}{\tilde{\sigma}} \approx_{o(1)} \expec{}{\rho \mid \bar{\tau}_1,\bar{\tau}_2}$.

Suppose that the contact network $G$ is realized and given as the problem input. we prove that the conditional expectation of infections given the contact network $G$ satisfies $\expec{\widetilde{G}}{\tilde{\sigma} \mid G}\approx_{o(1)}\expec{\widetilde{G}}{\rho \mid  G,\bar{\tau}_1,\bar{\tau}_2 }$ with high probability.

\begin{theorem}
\label{thm:rand_sub}
Suppose the contact network $G$ is generated from a random graph model $\calG(n,p_1)$. If the number of seeds $s = O(n^{\frac{1}{3}-c})$ for any constant $c$, then with high probability, the expected number of infections satisfies that for all $P\subseteq Q$
\begin{align*}
\expec{\widetilde{G}}{\tilde{\sigma}(P) \mid G} & \approx_{o(1)}
\expec{\widetilde{G}}{\tilde{\sigma}'(P) \mid G}
\end{align*}
where $\expec{\widetilde{G}}{\tilde{\sigma}'(P) \mid G} \defeq \expec{\widetilde{G}}{\rho(P) \mid  G,\bar{\tau}_1,\bar{\tau}_2}$ is a monotone supermodular function.
\end{theorem}

Theorem~\ref{thm:rand_sub} shows that $\expec{\widetilde{G}}{\sigma(P) \mid G}$
can be effectively minimized by a greedy algorithm. In this case, Algorithm~\ref{alg:greedy} takes as inputs the function $\expec{\widetilde{G}}{\tilde{\sigma}'(P) \mid G}$, a candidate set $Q$, initial states of nodes, and an integer $k$. It returns a set $P$ such that $\expec{\widetilde{G}}{\tilde{\sigma}(\emptyset) \mid G}-\expec{\widetilde{G}}{\tilde{\sigma}(P)\mid G}\geq
(1-1/e-o(1)-\eps)\left(\expec{\widetilde{G}}{\tilde{\sigma}(\emptyset)\mid G}-\expec{\widetilde{G}}{\tilde{\sigma}(P^*)\mid G}\right)$ holds with high probability, where $P^*$ is the optimum solution of $\expec{\widetilde{G}}{\tilde{\sigma}(P)\mid G}$. Again, the computation of $\expec{\widetilde{G}}{\tilde{\sigma}'(P)\mid G}$ is discussed in Section~\ref{dis:effiency}.

\begin{proof}
In Theorem~\ref{thm:ic_approx}, we show that
\begin{align}\label{eqn:unconditional}
\expec{}{\tilde{\sigma}} - \expec{}{\rho} \leq \expec{}{\tilde{\sigma}} - \expec{}{\rho \mid \bar{\tau}_1,\bar{\tau}_2} \prob{}{\bar{\tau}_1,\bar{\tau}_2} \leq \frac{sL^2d^3}{2n(1-d)} + \frac{2s^3L^2}{n} + \frac{1}{n}.
\end{align}
By the law of total expectation, we know that
\begin{align}\label{eqn:total}
    \expec{}{\tilde{\sigma}}-\expec{}{\rho}
    =& \expec{G}{\expec{\widetilde{G}}{\tilde{\sigma} \mid G}-\expec{\widetilde{G}}{\rho \mid G}} \\
    =& \expec{G}{\expec{\widetilde{G}}{\tilde{\sigma} \mid G}-\expec{\widetilde{G}}{\rho \mid G, \bar{\tau}_1, \bar{\tau}_2}\prob{\widetilde{G}}{\bar{\tau}_1,\bar{\tau}_2 \mid G} } \nonumber\\
    & -\expec{}{\rho \mid \bar{\tau}_1,\tau_2}\prob{}{\tau_2,\bar{\tau}_1}
    -\expec{}{\rho \mid \tau_1}\cdot \prob{}{\tau_1} \nonumber\\
    \geq& \expec{G}{\expec{\widetilde{G}}{\tilde{\sigma} \mid G}-\expec{\widetilde{G}}{\rho \mid G, \bar{\tau}_1, \bar{\tau}_2}\prob{\widetilde{G}}{\bar{\tau}_1,\bar{\tau}_2 \mid G} } -\frac{1}{n}- \frac{2s^3L^2}{n}, \nonumber
\end{align}
where the last inequality is due to Lemma~\ref{lem:cc_size} and Lemma~\ref{lem:seed_collision}.

Let the random variable $Z$ be $\expec{\widetilde{G}}{\tilde{\sigma} \mid G}-\expec{\widetilde{G}}{\rho \mid G, \bar{\tau}_1, \bar{\tau}_2}\prob{\widetilde{G}}{\bar{\tau}_1,\bar{\tau}_2 \mid G}$. By combining two inequalities \eqref{eqn:unconditional} and \eqref{eqn:total}, we have
\begin{align*}
\expec{G}{Z} \leq  \expec{}{\tilde{\sigma}}-\expec{}{\rho} + \frac{1}{n} + \frac{2s^3L^2}{n} \leq \frac{sL^2d^3}{2n(1-d)} + 2\left(\frac{2s^3L^2}{n} + \frac{1}{n}\right).
\end{align*}
Setting $\lambda_0 = s^3/n^{1-c}$ for any constant $c > 0$, by 
Markov's inequality, we have
\begin{align*}
\prob{}{Z\geq \lambda_0} \leq \frac{\expec{G}{Z}}{\lambda_0} \leq \frac{L^2d^3}{2n^c(1-d)s^2} + \frac{6L^2}{n^c},
\end{align*}
which implies, with high probability, 
that
\begin{align*}
0 \leq \expec{\widetilde{G}}{\tilde{\sigma} \mid G}-\expec{\widetilde{G}}{\rho \mid G, \bar{\tau}_1, \bar{\tau}_2}\prob{\widetilde{G}}{\bar{\tau}_1,\bar{\tau}_2 \mid G} \leq \lambda_0.
\end{align*}

To show the approximation result, we need to show that $\prob{\widetilde{G}}{\bar{\tau}_1,\bar{\tau}_2 \mid G}$ is close to $1$ with high probability. By Lemma~\ref{lem:cc_size} and Lemma~\ref{lem:seed_collision}, we have
\begin{align*}
\expec{G}{1-\prob{\widetilde{G}}{\bar{\tau}_1,\bar{\tau}_2 \mid G}} = 1- \prob{}{\bar{\tau}_1,\bar{\tau}_2} = \prob{}{\tau_1} + \prob{}{\tau_2,\bar{\tau}_1} \leq \frac{1}{n^2} + \frac{2s^2L}{n}. 
\end{align*}
By using Markov's inequality, we have
\begin{align*}
\prob{}{1-\prob{\widetilde{G}}{\bar{\tau}_1,\bar{\tau}_2 \mid G} \geq \lambda_0} \leq \frac{\expec{G}{1-\prob{\widetilde{G}}{\bar{\tau}_1,\bar{\tau}_2 \mid G}}}{\lambda_0} \leq \frac{3L}{sn^c},
\end{align*}
which implies that $\prob{\widetilde{G}}{\bar{\tau}_1,\bar{\tau}_2 \mid G} \geq 1-\lambda_0$ with high probability.

Therefore, by using the union bound, we conclude, 
with probability at least $1-O(L^2/n^c)$,
that
\begin{align*}
\expec{\widetilde{G}}{\tilde{\sigma}(P) \mid G} & \approx_{\lambda_0}\expec{\widetilde{G}}{\rho(P) \mid  G,\bar{\tau}_1,\bar{\tau}_2},
\end{align*}
where $\lambda_0 = s^3/n^{1-c}$ for any constant $c>0$. Note that $\expec{\widetilde{G}}{\rho(P) \mid  G,\bar{\tau}_1,\bar{\tau}_2}$ is a supermodular function of the set of removed edges $P$.
\end{proof}

\section{Supermodularity for the IC-SIR model in SBMs}
\label{sec:IC_SBM}

In this section, we consider the IC-SIR model on the network generated from SBMs. We assume there are $\kappa = O(\ln n)$ communities with $n$ nodes. The set of seeds is still fixed before the contact network and the contagion network are generated. 

In Section~\ref{sbm:sec} we investigate the expected number of infections in the IC-SIR model where the contact network is complete and the contagion network $\widetilde{G}$ is sampled from $\SBM(n,\kappa,\QQ)$. In Section~\ref{sbm_sample:sec} we investigate the case where the contact network $G$ is generated from $\SBM(n,\kappa,\QQ)$ and the contagion network $\widetilde{G}$ is again sampled from $G$ with probability $p$ for each edge.

\subsection{Expected Number of Infections in SBMs}
\label{sbm:sec}
Given a complete contact network, we consider the contagion network $\widetilde{G}$ generated from  $\SBM(n,\kappa,\QQ)$, where $\kappa=O(\ln n)$, matrix $\QQ$ is symmetric and $\QQ\1$ is entry-wise less than $\frac{1}{n}\1$. In this case, we prove that the expected number of infections $\tilde{\sigma}(P)$ over the contagion network $\widetilde{G}$ is a $(1\pm o(1))$ approximation of a supermodular function. For simplicity, we define $d_{ij} = nQ_{ij}$ as the expected degree between block $i$ and block $j$. We further let $d_{\text{init}}$ be the maximum intra-block expected degree, defined as $ d_{\text{init}} =  n\cdot \max\left( \diag{\QQ}\right)$, and $d_{\text{end}}$ be the maximum expected degree, defined as $ d_{\text{end}} = n\cdot \max\left( \QQ\1\right)$. 

Similar to the analysis of ER
random graphs, we define $\tau_1$ as the event that the graph includes connected components with size greater than $L^*=9(1-d_{\rm{end}})^{2}\ln(n\kappa)$. Let $\tau_2$ be the event that there is at least one component with more than one seed. Let events $\bar{\tau}_1$ and $\bar{\tau}_2$ be their complements respectively. Since the set of seeds is fixed before the generation of the contagion network, the seeds in each community can be seen as sampled from each community uniformly at random afterwards. 

\begin{theorem}
\label{thm:sbm}
If the number of blocks $\kappa=O(\ln{n})$, the number of seeds $s=O(n^{\frac{1}{3}-c})$, and $d_{\rm{end}}<1-c$ for any absolute constant $c>0$, then for all $P\subseteq Q$,
\begin{align*}
\expec{\widetilde{G}}{\tilde{\sigma}(P)} &\approx_{o(1)} \expec{\widetilde{G}}{\tilde{\sigma}'(P)},
\end{align*}
where $\expec{\widetilde{G}}{\tilde{\sigma}'(P)} \defeq \expec{\widetilde{G}}{\rho(P) \mid \bar{\tau}_1,\bar{\tau}_2}$ is a monotone supermodular function of the edge deletion set $P$. 
\end{theorem}

Theorem~\ref{thm:sbm} shows that $\expec{\widetilde{G}}{\tilde{\sigma}(P)}$ is a close approximation to a monotone supermodular function. Therefore, similar to the result in Section~\ref{random.sec}, a greedy algorithm returns a set $P$ which gives a $(1-1/e-o(1)-\eps)$-approximation to the optimum decrease of expected number of infections.

We start by proving that $\tau_1$ happens with a small probability.
\begin{lemma}
\label{lem:sbm_cc_size}
The probability that the contagion network $\widetilde{G}$ has a connected component with size greater than $L^*$ is bounded by
\begin{align*}
    \prob{}{\tau_1} \leq n^{-2}\,.
\end{align*}
\end{lemma}
\begin{proof}
The proof is the same the proof of Lemma~\ref{lem:cc_size} except that the number of new leaves $Z_t$ is the sum of $n\kappa$ independent Bernoulli random variables with various probabilities of success, instead of \emph{i.i.d.} Bernoulli random variables. The result follows from the Chernoff bound by using $d_{\rm{end}}$ to upper bound the expected degree of each node.
\end{proof}

We also obtain the following lemma which is almost the same as Lemma~\ref{lem:seed_collision}.
\begin{lemma}
The probability that there is at least one component with more than one seed conditioned on event $\bar{\tau}_1$ is 
\begin{align*}
    \prob{}{\tau_2 \mid \bar{\tau}_1} \leq \frac{2s^2L^*}{n},
\end{align*}
where $s$ is the number of seeds.
\end{lemma} 
\begin{proof}
We omit the proof since the approach is identical to that of Lemma~\ref{lem:seed_collision}.
\end{proof}
The key to the proof of Theorem~\ref{thm:sbm} is an upper bound for the expected number of nodes in connected components with cycles, which is given by the following lemma.

\begin{lemma}
\label{lem:sbm_cycles}
Given that the events $\bar{\tau}_1$ and $\bar{\tau}_2$ happen, the expected number of vertices in connected components which contain at least one seed and one cycle satisfies
\begin{align*}
    \expec{}{y(P) \mid \bar{\tau}_1, \bar{\tau}_2} \prob{}{\bar{\tau}_1,\bar{\tau}_2}= O\left(\frac{s(\ln n)^{4m_{\rm{bin}}}}{n}\right),
\end{align*}
where $m_{\rm{bin}}\defeq 4\left\lceil(1-d_{\rm{init}})^2/(1-d_{\rm{end}})^2\right\rceil$.
\end{lemma}

The proof of Lemma~\ref{lem:sbm_cycles} 
is in Appendix~\ref{apx:IC_SBM}.

\begin{proof}[Proof of Theorem~\ref{thm:sbm}]
Similar to the proof of Theorem~\ref{thm:ic_approx}, we use the bounds given by Lemmas~\ref{lem:sbm_cc_size},~\ref{lem:seed_collision}, and~\ref{lem:sbm_cycles} to prove that $\expec{\widetilde{G}}{\tilde{\sigma}(P)} \approx_{o(1)} \expec{\widetilde{G}}{\rho(P) \mid \bar{\tau}_1,\bar{\tau}_2}$. The monotonicity and supermodularity of $\expec{\widetilde{G}}{\rho(P) \mid \bar{\tau}_1,\bar{\tau}_2}$ again follows from the fact that for all instances of $\widetilde{G}$ where $\bar{\tau}_1$ and $\bar{\tau}_2$ happen, the set function $\rho(P)$ is monotone supermodular.
\end{proof}

\subsection{Expected Number of Infections for a Random Contact Network in SBMs}
 \label{sbm_sample:sec}   
 
In this section, we assume the contact network $G$ is generated from $\SBM(n,\kappa,\QQ)$, and the contagion network $\widetilde{G}$ is again sampled from $G$ with probability $p$ for each edge.  Thus, the contagion network $\widetilde{G}$ can be seen as generated from $\SBM(n,\kappa,p\QQ)$. We assume that the probability matrix $\QQ$ satisfies that $p\QQ\1$ is entry-wise less than $\frac{1}{n}\1$. In this case, the maximum expected degree $d_{\text{end}} = n\cdot \max\left( p\QQ\1\right)$ is less than $1$. 
 
We obtain a result similar to Theorem~\ref{thm:rand_sub} for a given contact network generated from $\SBM(n,\kappa,\QQ)$.

\begin{theorem}
\label{thm:sbm_sub}
Suppose the contact network $G$ is generated from a stochastic block model $\SBM(n,\kappa,\QQ)$. If the number of blocks $\kappa=O(\ln{n})$, the number of seeds $s = O(n^{\frac{1}{3}-c})$, and $d_{\rm{end}}<1-c$ for any absolute constant $c>0$, then with high probability, the expected number of infections satisfies that for all $P\subseteq Q$,
\begin{align*}
\expec{\widetilde{G}}{\tilde{\sigma}(P) \mid G}\approx_{o(1)}\expec{\widetilde{G}}{\tilde{\sigma}'(P) \mid G},
\end{align*}
where $\expec{\widetilde{G}}{\tilde{\sigma}'(P) \mid G} \defeq \expec{\widetilde{G}}{\rho(P) \mid  G,\bar{\tau}_1,\bar{\tau}_2}$ is a monotone supermodular function of the edge deletion set $P$.
\end{theorem}
\begin{proof}
The proof is omitted due to the similarity to the proof of Theorem~\ref{thm:rand_sub}.
\end{proof}

Similar to Section~\ref{sec:ic_random_contact}, we conclude that a greedy algorithm can be used to effectively minimize $\expec{\widetilde{G}}{\tilde{\sigma}(P) \mid G}$ with high probability.

\section{Discussion}
\label{discuss:sec}
In this section, we discuss the efficiency and the robustness of the proposed algorithm for the IC-SIR model and the computational complexity of the proposed optimization problems. In Section~\ref{dis:effiency}, we present an algorithm to efficiently approximate the expected number of infections in the IC-SIR model. In Section~\ref{dis:robustness}, we consider the performance of our algorithm for the case where adversarial nodes are present in the contact network. In Section~\ref{sec:hardness}, we show that Problems~\ref{prob:dm} and~\ref{prob:ic} are $\mathbf{NP}$-hard when the reproduction number is large.


\subsection{Approximating the Expected Number of Infections in the IC-SIR Model}
\label{dis:effiency}
In the IC-SIR model, we need to approximate the expected number of infections $\expec{}{\tilde{\sigma}(P)}$ and its $(1\pm o(1))$-approximation $\expec{}{\tilde{\sigma}'(P)}$ on a contact network. 
We do so by sampling contagion networks, similar to the approach used in \cite{BBCL14,Bha18}.
For a given set of seeds and a contact network, we repeatedly execute the sampling process defined as follows: 

We start by sampling a contagion network and an uniformly sampled terminal node. This process is defined as a success if this terminal node is reachable from the set of seeds. 
Among $R$ rounds of this process, let $R_s$ be the number of total successes. Then, we use $R_s n/R$ as the estimator of the expected number of infections $\expec{}{\tilde{\sigma}}$. 

\begin{proposition}
For $R = O(\epsilon^{-2}n\ln n)$ rounds of the sampling processes, we can approximate the expected number of infections $\expec{}{\tilde{\sigma}(P)}$ for all sets $P$ and its marginal gains upon edge deletion in the IC-SIR model within a factor of $(1\pm\epsilon)$ with high probability. 
\end{proposition}

\begin{proof}
For $R = O(\eps^{-2} n \ln n)$ rounds, by the Chernoff bound, the estimator $R_s n/R$ is a $(1\pm\eps)$-approximation to the expected number of infections with high probability. Then we consider updating the expected number of infections upon deleting a set $P$ of edges from the contact network. Since all terminals and contagion networks are mutually independent, we can use the same pairs of terminal node and contagion network with different $P$. We check the reachability of a node in the corresponding contagion network without using edges in $P$. For a greedy algorithm it suffices to know whether the current edge $e\in Q\backslash P$ is a bridge that separates the terminal and the seeds in the contagion graph to compute the marginal gains $\expec{}{\tilde{\sigma}(P)}-\expec{}{\tilde{\sigma}(P\cup\{e\})}$. Then we use the number of successes without using edges in $P$ and $P\cup\{e\}$ to estimate $\expec{}{\tilde{\sigma}(P)}$ and the marginal gains for all remaining candidate edges. Due to the Chernoff bound, the resulted expectations and the marginal gains are well approximated with high probability. The overall high probability guarantee follows by a union bound. 
\end{proof}

\begin{remark}
We can approximate the function $\expec{}{\tilde{\sigma}'(P)} = \expec{}{\rho(P) \mid \bar{\tau}_1, \bar{\tau}_2}$ for all sets $P$ and its marginal gains by modifying the sampling process as follows. If the sampled contagion network has large connected components or seeds collision corresponding to events $\tau_1,\tau_2$ respectively, then we resample the contagion network. If the terminal node is contained in a connected component with cycles, then this process is seen as a failure. 
\end{remark}


\subsection{Robustness of the Algorithm for the IC-SIR Model}
\label{dis:robustness}
We discuss the robustness of the algorithm in the IC-SIR model, in the presence of $O(\ln{n})$ oblivious adversarial nodes,
which add incident edges to the graph without information about the seeds and the graph except for its direct neighbors.
To this end, we consider a contact network generated from an ER graph or a SBM with adversarial perturbation. After the generation of the contact network, the adversarial nodes randomly augments their connections to other nodes. In the following, we give an example of such an adversarial model.

\begin{example}
Consider a contact network constructed as follows: let $G_0=(V,E_0)$ be a random graph generated from $\calG(n, p_1)$ where $p_1 = (\ln n) / n$. Then, we choose $O(\ln n)$ adversarial nodes and add $\nu$ non-existing edges to each of them uniformly at random. Let $E_1$ be the set of edges added. Thus, the contact network is $G=(V,E_0\cup E_1)$. The contagion network $\widetilde{G}$ is sampled from $G$ with probability $p < 1/\ln n$.
Then, we consider the following two cases.

If the number of edges added to each adversarial node $\nu \geq 4\ln n$, then by the Chernoff bound, the degree of a node in a random graph is at most $4\ln n$ with probability at least $1-(1/n)^{4/5}$. Thus, we can identify these adversarial nodes easily in this case by checking the degree of each node in the contact network. In this case we need to reduce the degree of these nodes to at most $4\ln n$, which will reduce the network to the following situation.

If the number of edges added to each adversarial node $\nu < 4\ln n$, then these adversarial nodes could potentially connect several connected components in $\widetilde{G}$. We note that the degree for any adversarial nodes after the edge addition is at most $8\ln n$ with high probability. Given that $\widetilde{G}-(V,E_1)$ has only connected components of sizes $O(\ln{n})$, these adversarial nodes can add a connected component with size at most $O(\ln^3{n})$. Since the number of seeds is $n^{\frac{1}{3}-c}$ for any $c>0$, the probability that the aforementioned connected component has at least one seed is $O(n^{-\frac{2}{3}-c}\cdot \ln^3{n})$. 
Therefore, in the worst case, it will induce an $O(n^{-\frac{2}{3}-c}\cdot \ln^6{n}) = o(1)$ expected difference between $\tilde{\sigma}$ and $\rho$.
\end{example}



\subsection{Hardness of the Optimization Problems}
\label{sec:hardness}
We prove that Problems~\ref{prob:dm} and~\ref{prob:ic} are $\mathbf{NP}$-hard in a regime where the reproduction number is greater than one. We note that the hardness of these problems with reproduction number less than or equal to one remains open.

For the D-SIR model, we consider the
case where 
\begin{align}
\label{inst:DMhard}
D_{i}=0, \forall i\in V;\quad  B_{ij}>0, \forall (j,i)\in E.
\end{align}
For the IC-SIR model, we consider the case where
\begin{align}
\label{inst:IChard}
    B_{ij}=1, \forall (j,i)\in E,
\end{align} such that the contagion network is equivalent to the contact network. With these parameters, both Problems~\ref{prob:dm} and~\ref{prob:ic} become equivalent to the graph partitioning problem formulated in~\cite{enns2012optimal}.

\begin{theorem}
\label{thm:hard}
Given an instance of Problem~\ref{prob:dm} satisfying condition~(\ref{inst:DMhard})  or Problem~\ref{prob:ic} satisfying condition~(\ref{inst:IChard}), there does not exist a polynomial time algorithm which finds an optimal solution $P^*$ unless $\mathbf{P}=\mathbf{NP}$.
\end{theorem}

In particular, we study the following decision problem.
\begin{problem}
\label{prob:decision}
Given a graph $G=(V,E)$ with $n$ nodes, a  vector $\ss\in\{0,1\}^n$ which defines the set of seeds $\calS=\mathrm{supp}(\ss)$, a candidate edge deletion 
set $Q\subseteq E$ with $|Q|=q$, two integers $0<k\leq q$ and $0\leq z\leq n$, decide whether or not there exists an edge set $P$, with $|P|=k$, such that after removing edges in $P$, the number of nodes that are reachable from the seeds is less or equal to $z$.
\end{problem}

The hardness of Problem~\ref{prob:decision} follows from the $\mathbf{NP}$-completeness of \emph{minimum bisection of $3$-regular graphs}.
\begin{lemma}[Theorem 2.12 in~\cite{Bui86}]
\label{lemma:MinBisec}
The problem of deciding whether or not a $d$-regular graph has a bisection of size $b$ or less is $\mathbf{NP}$-complete, whenever $d\geq 3$ and $b=n^{\eps}$ for any fixed $\eps\in(0,1)$.
\end{lemma}

\begin{proof}[Proof of Theorem~\ref{thm:hard}]
We construct an instance as follows. For a $3$-regular graph $\widehat{G}=(V,E)$ with $n$ nodes where $n$ is a even number, we add $3$ copies of a star graph $G^*$ to it, denoted $G^*_{(1)}=(V_1,E_1)$, $G^*_{(2)}=(V_2,E_2)$, and $G^*_{(3)}=(V_3,E_3)$. Each one of the $(n+1)$-node star graphs has $n$ leaves supported on $V$. The central nodes of the $3$ stars are denoted as $v_{(1)}$, $v_{(2)}$, and $v_{(3)}$, respectively.

Then we construct the graph 
$G'= \widehat{G}+G^*_{(1)}+G^*_{(2)}+G^*_{(3)}$. We consider an instance of Problem~\ref{prob:decision} with the input $G=G'$, $\calS = \{v_{(1)}, v_{(2)}, v_{(3)}\}$, $Q=E\cup E_1\cup E_2\cup E_3$, $k=b+\frac{3n}{2}$, and $z=3+\frac{n}{2}$.

Completeness: 
If there is a bisection of $\widehat{G}$ with size $b$, then by definition there is a partition of nodes $S$ and $\bar{S}\defeq V\backslash S$, such that the number of edges in between is $b=|C(S, \bar{S})|$, and $|S|=|\bar{S}|=n/2$. By cutting $b$ edges between $S$ and $\bar{S}$, and all the edges $(u, v_{(i)})$, for all $u\in S$ and $i\in\{1,2,3\}$, we separate $n/2$ nodes from the rest of the graph. Therefore the number of nodes reachable from the seeds is equal to $3+\frac{n}{2}$.

Soundness: It suffices to prove that if there is no bisection of size $b$ or less for $\widehat{G}$, then there is no such set $P$, with
$|P|=b+\frac{3n}{2}$, such that by removing edges in $P$, $\frac{n}{2}$ nodes are no longer reachable from the seeds. We prove it by contradiction. We assume the bisection of size $b$ for $\widehat{G}$ does not exist. If there exists such a 
set $P$ for the constructed instance of Problem~\ref{prob:decision}, then $P$ contains a cut $C(S,\bar{S})$ of graph $G'$ with at most $b+\frac{3n}{2}$ edges, with 
$|S|\geq n/2$ and $S \cap \calS=\emptyset$. If $|S|=n/2$ then $P\cap E$ contains a bisection of $\widehat{G}$ of size less than or equal to $b$. Therefore we have constructed a contradiction. If $|S|>n/2$, one can always find an vertex $u\in S$, such that $|C(S\backslash\{u\}, \bar{S}\cup \{u\})|\leq b+\frac{3n}{2}$, since $u$ has at most $3$ neighbors in $S$ and at least $3$ neighbors in $\bar{S}$. By repeating this process we arrive at a new cut $C'\defeq C(S', \bar{S'})$ of $G'$ such that $|C(S', \bar{S'})|\leq b+\frac{3n}{2}$, $|S|=n/2$, and $S \cap \calS=\emptyset$. Then $C' \cap E$ is a bisection of $\widehat{G}$ with size $b$ or less. Again, we obtain a contradiction, concluding
the proof.
\end{proof}

\section{Numerical Results}
\label{sec:numerical}

In this section we run experiments to show the effectiveness of the proposed algorithms in all considered models. We first study the effectiveness of our greedy algorithms in the D-SIR and IC-SIR model by comparing it to two baseline algorithms, Random~\cite{callaway2000network} and Max-Degree~\cite{albert2000error}. 
\begin{itemize}
    \item Greedy: choose an edge greedily at each round with respect to the objective function, as shown in Algorithm~\ref{alg:greedy}.
    \item Max-Degree: remove an edge incident to the node with the maximum degree at each round. 
    \item Random: remove edges chosen uniformly at random
    from the network. 
\end{itemize}


\subsection{D-SIR Simulations}\label{subsec:dsir_sim}

\begin{figure}
\centering
\includegraphics[width=0.4\textwidth]{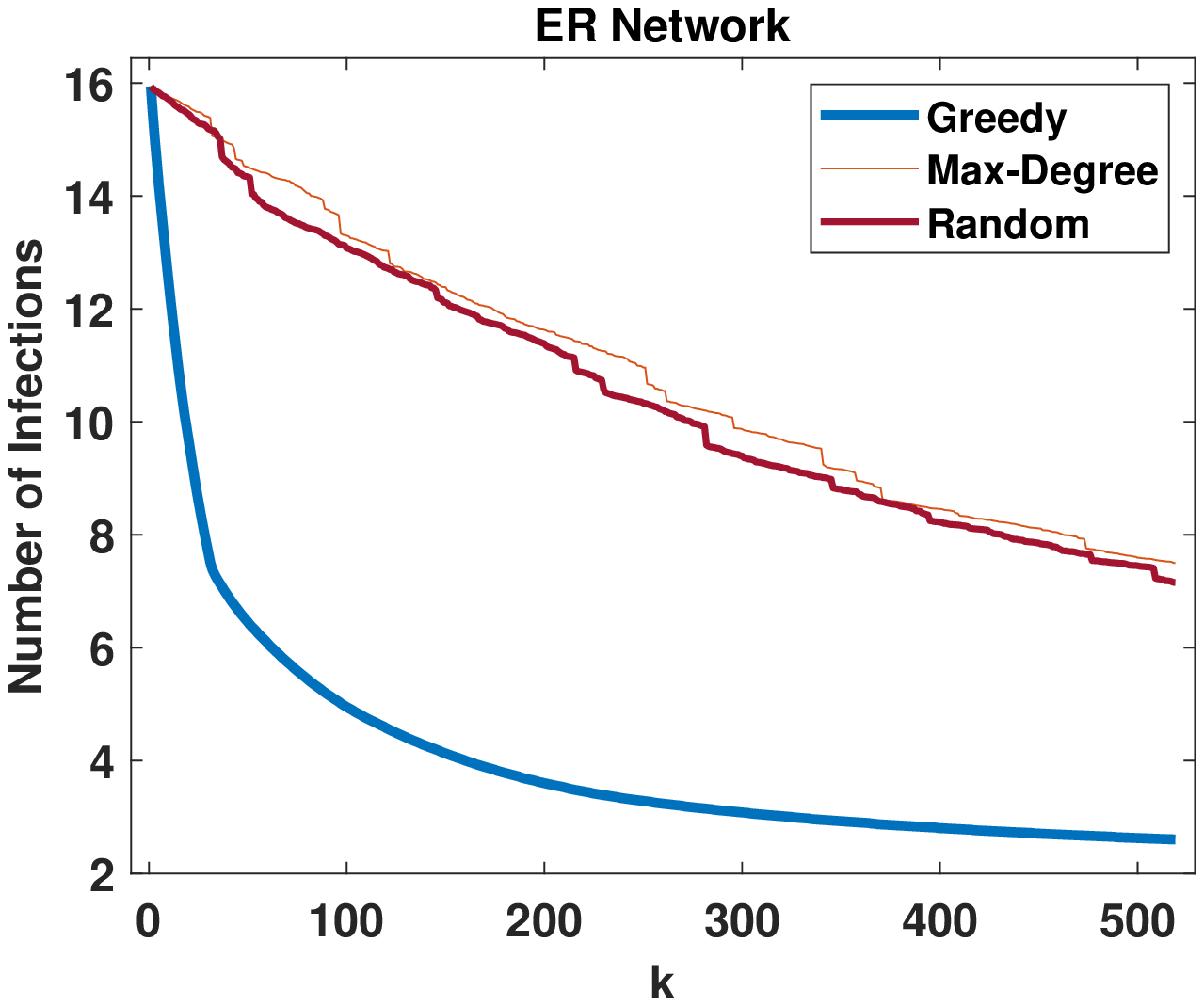}\quad
\includegraphics[width=0.4\textwidth]{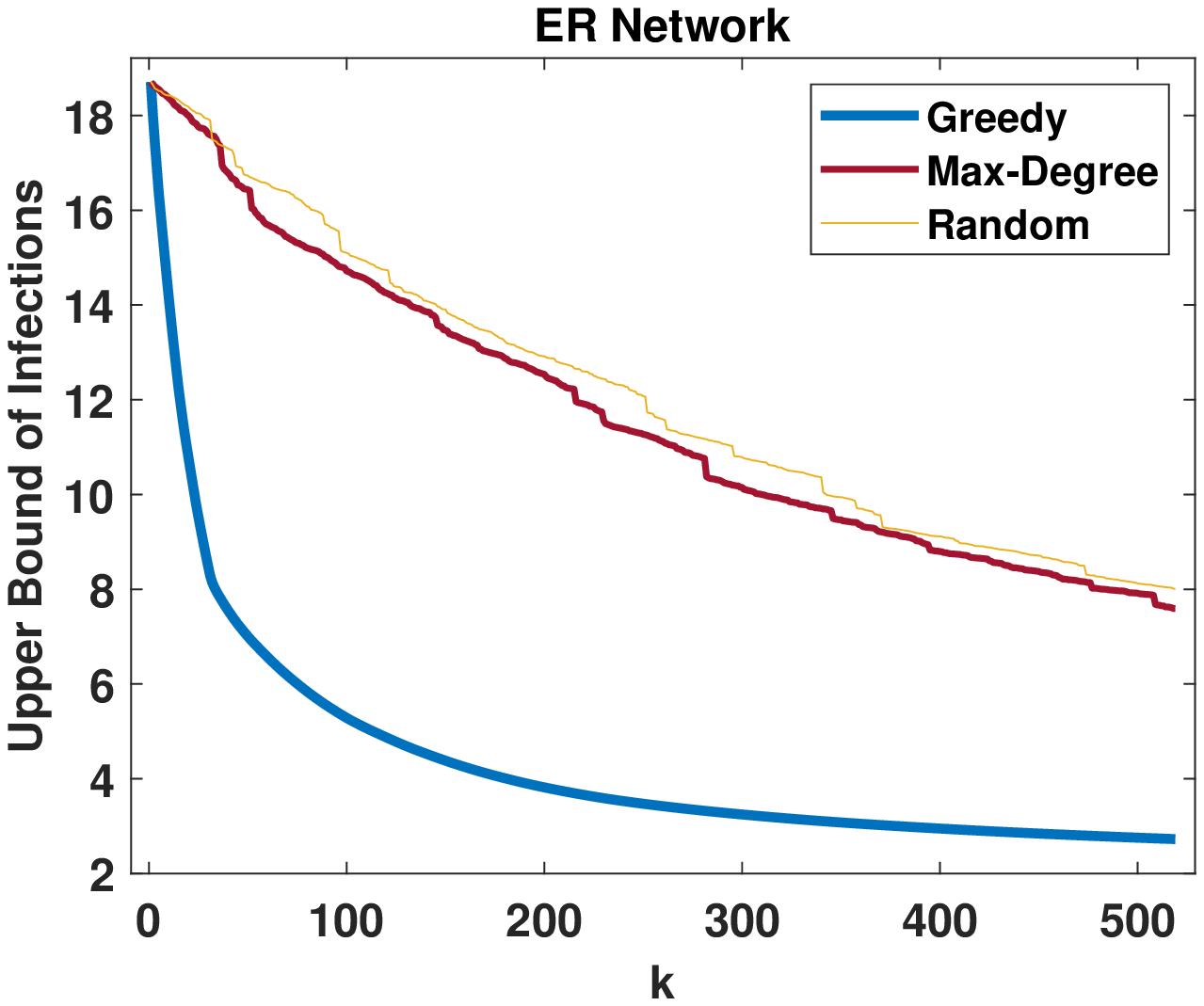}
\includegraphics[width=0.4\textwidth]{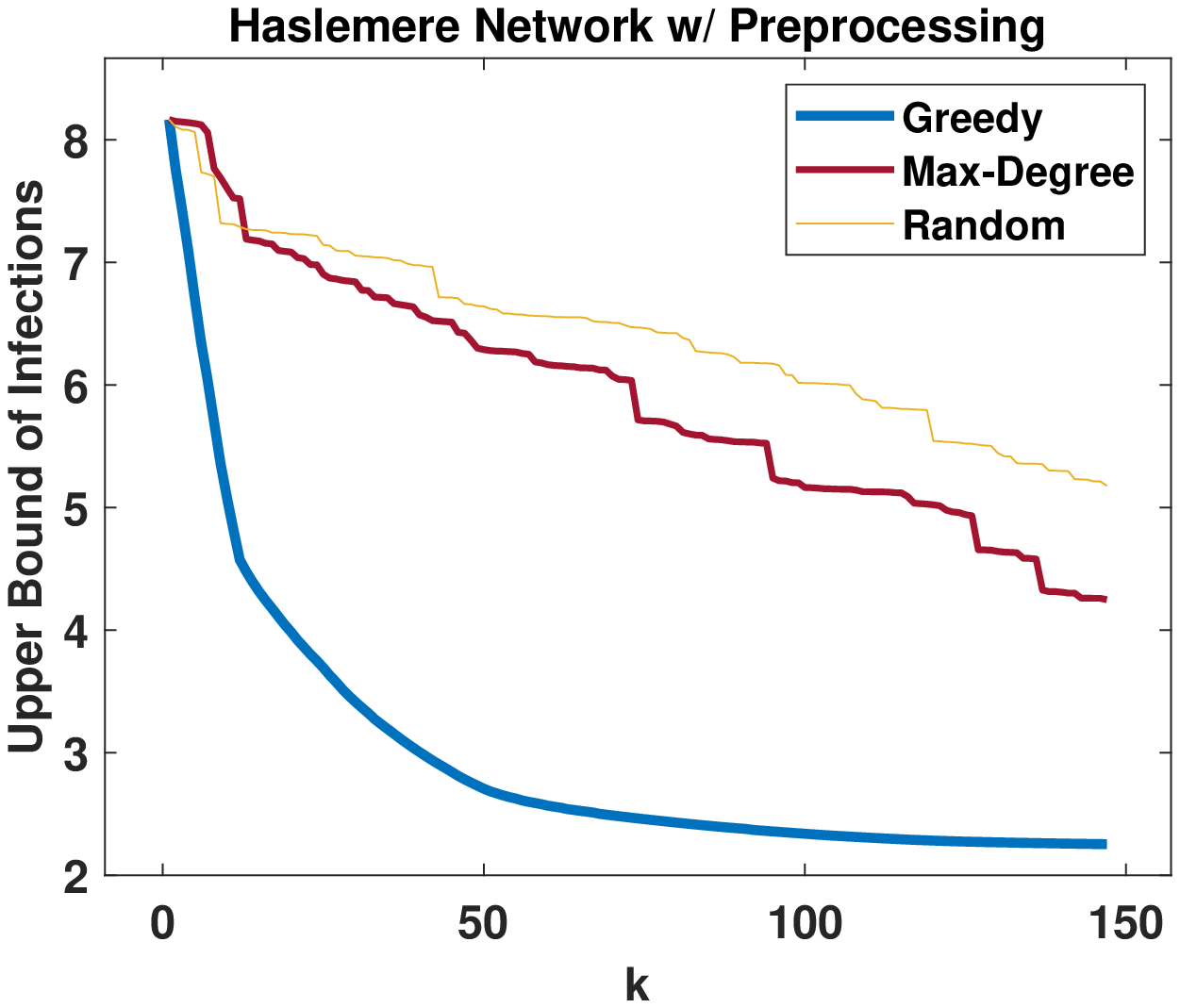}\quad
\includegraphics[width=0.4\textwidth]{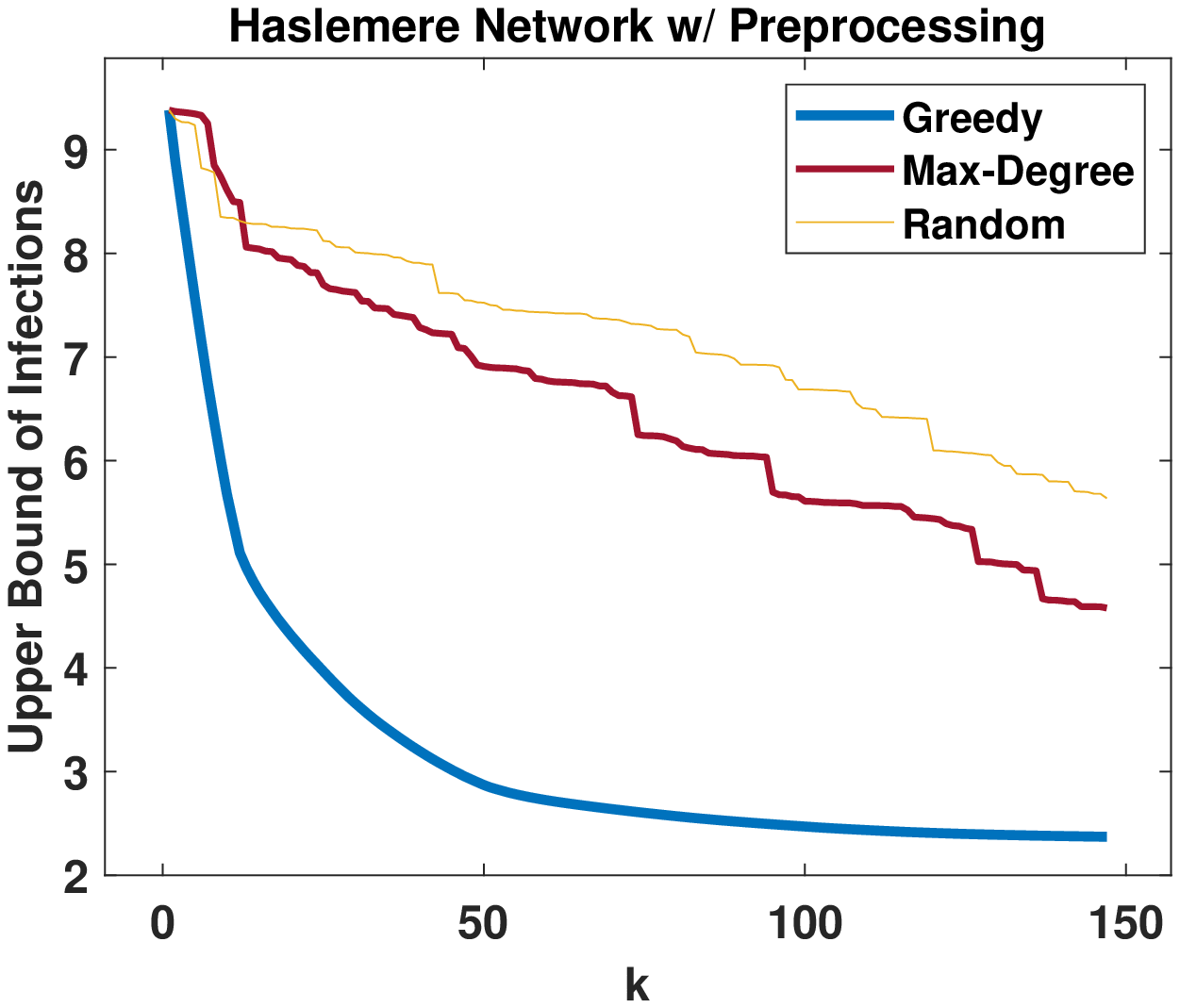}
\caption{In the D-SIR model, the number of infections $\sigma$ and the proposed upper bound of infections $\hat{\sigma}$ for the edge deletion sets given by Greedy, Max-Degree, and Random algorithms on the ER network and the preprocessed Haslemere contact network. The title of the horizontal axis $k$ is the number of removed edges. }
\label{fig:D-SIR}
\end{figure}

For the D-SIR model, we run experiments on an ER random graph with $500$ nodes and the connection probability $0.0249$. The edges in the network are always bidirectional, with 
possibly different
infection rates for the two directions. For a generated graph $G(V,E)$ with $3142$ edges, we uniformly at random
pick the recovery rate $D_{i}\in[0.28,0.35]$ for each $i$ and the infection rate $B_{ij}\in[0.011,0.034]$ for each $i,j$ pair. Then, we 
uniformly at random
sample $5$ nodes from the network as the set of seeds $\mathcal{S}$, whose initial infection probabilities, $x_u(0)$ for all $u\in \mathcal{S}$, are 
uniformly at random
chosen from $[0.8,0.9]$. For any other node $u\notin \mathcal{S}$, its infection probability is $x_u(0)=0$.  Each node $v\in V$ has an initial removed probability $r_v(0)$ uniformly at random
chosen from $[0,0.05]$. The candidate set $Q$ has 
$1571$ edges which are 
randomly and uniformly
chosen from the edge set $E$. 
Further, we run three algorithms to choose the edge deletion set $P \subseteq Q$ satisfying $|P|\leq k$ for $k=523$. We show the results of the three algorithms for both the number of infections $\sigma$ and the proposed upper bound of the number of infections $\hat{\sigma}$ in Figure~\ref{fig:D-SIR}.



Next we simulate the dynamics over a real contact network collected in Haslemere, England using mobile technologies
~\cite{klepac2018contagion,firth2020using}.
In this example, the contact network has $469$ nodes and $1262$ edges, with the maximum degree equal to $37$. For this contact network, the condition for the exponential stable in Theorem~\ref{sufficient:theorem} can only be guaranteed with small infection rates and large recovery rates. To address a more realistic setting, we preprocess the contact network by greedily removing some edges from the high degree nodes as 
the Max-Degree algorithm does. 
This constraint on the max degree of nodes can be implemented in reality by posing restrictions on gathering~\cite{Sweden_order} or shelter-in-place orders \cite{California_order}. The preprocessed network has $877$ edges, with the maximum degree equal to $8$. 
Then we set $D_i=0.5$ for each $i$ and $B_{ij}\in[0.056,0.063]$ for each $i,j$ pair. We choose the candidate set $Q$ with $518$ edges 
uniformly at random
from the edge set. The results of three algorithms with $k=219$ on the preprocessed network are shown in Figure~\ref{fig:D-SIR}. 

The results clearly show that our algorithm outperforms 
the two heuristic algorithms in minimizing the number of infections on both the ER network and the Haslemere contact network. We can also observe that the results for the upper bound $\hat{\sigma}$ and the expected number of infections $\sigma$ behave consistently in these examples.



\subsection{IC-SIR Simulations}

For the IC-SIR model, we implement three algorithms on three different networks, the ER network, the SBM network, and the Haslemere network with preprocessing. The approximations of the expected number of infections after deleting $k$ edges returned by three algorithms are shown in Figure~\ref{fig:IC_SIR}.

\begin{figure}
\centering
\includegraphics[width=0.4\textwidth]{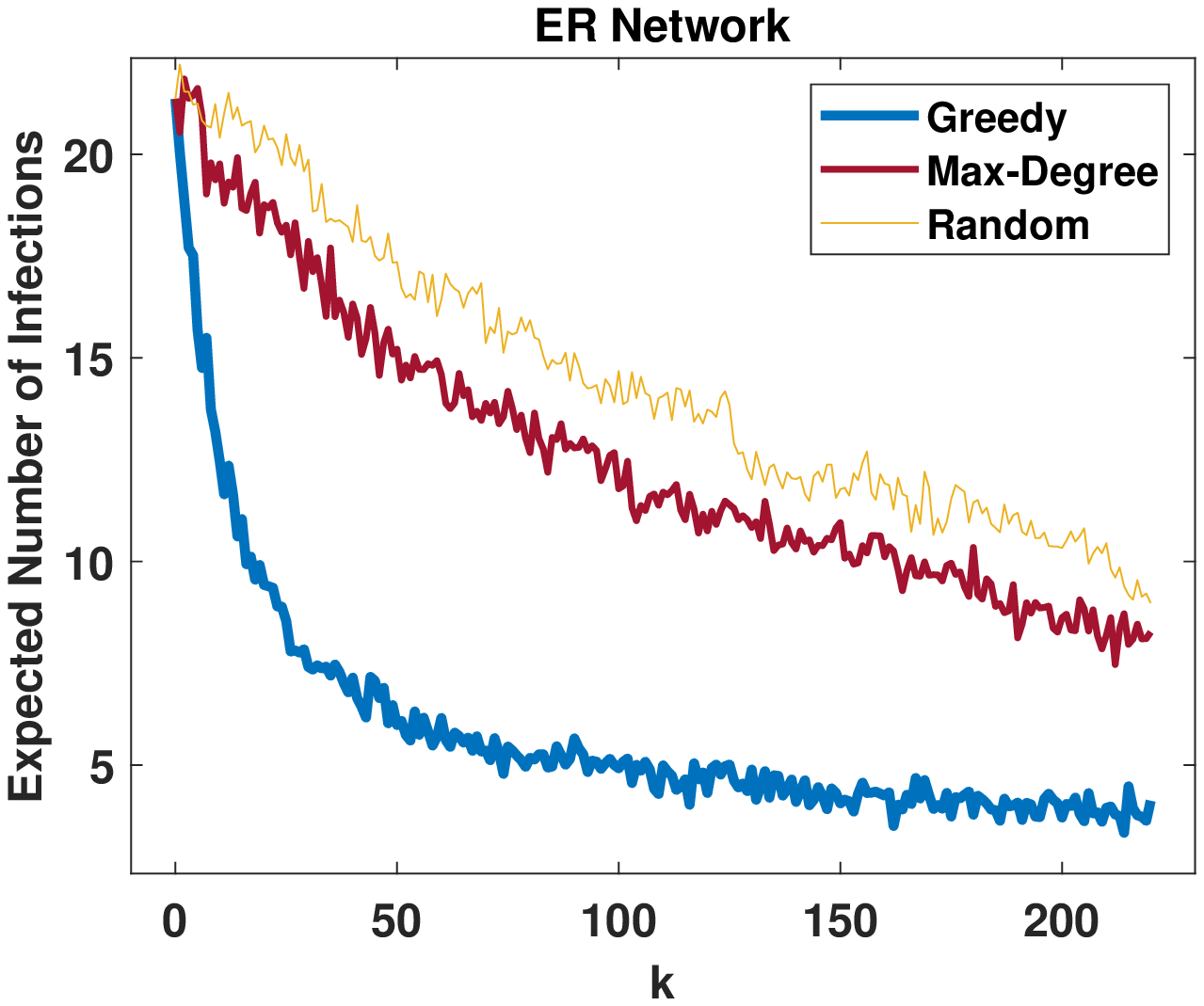}\quad
\includegraphics[width=0.4\textwidth]{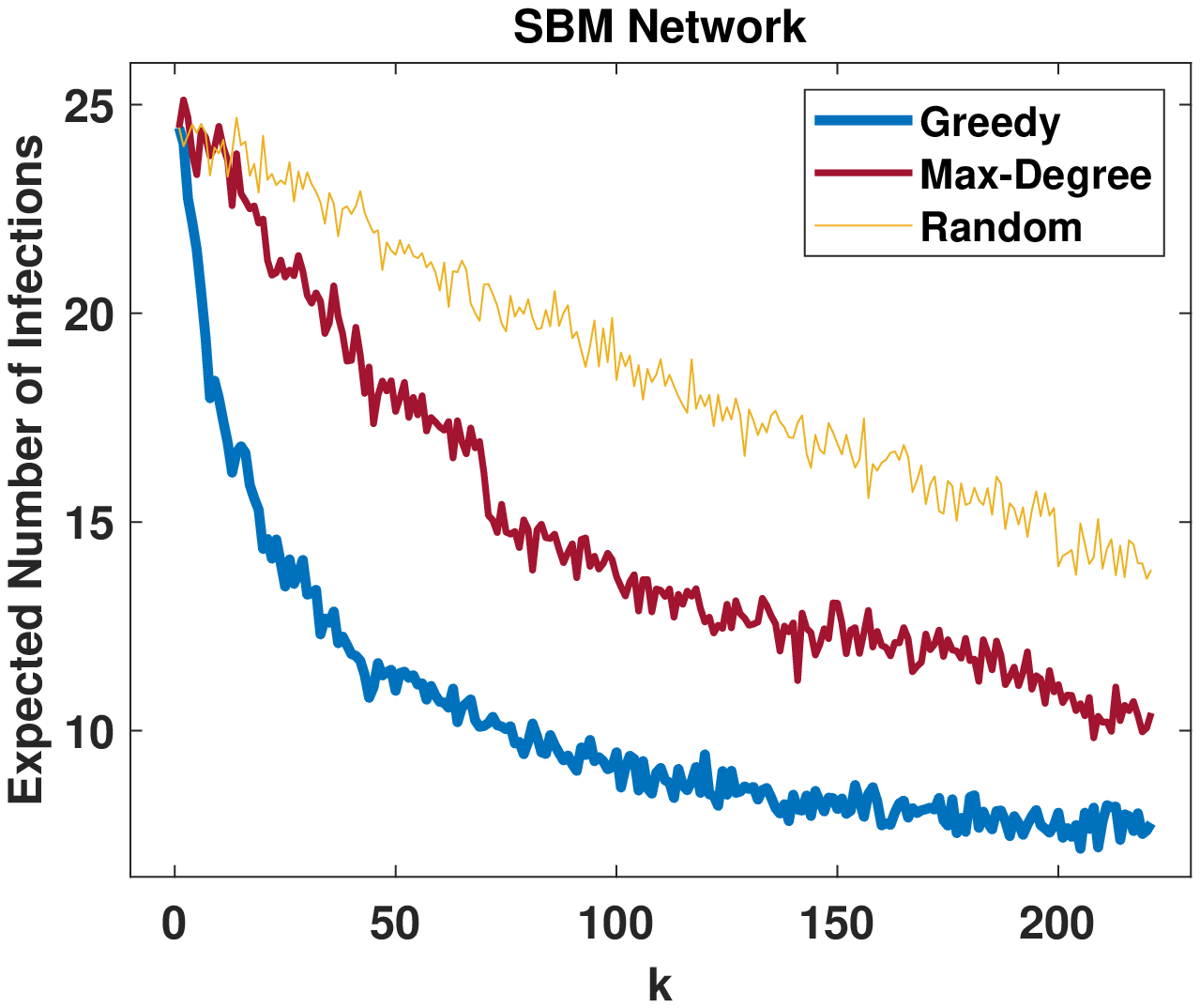}
\includegraphics[width=0.4\textwidth]{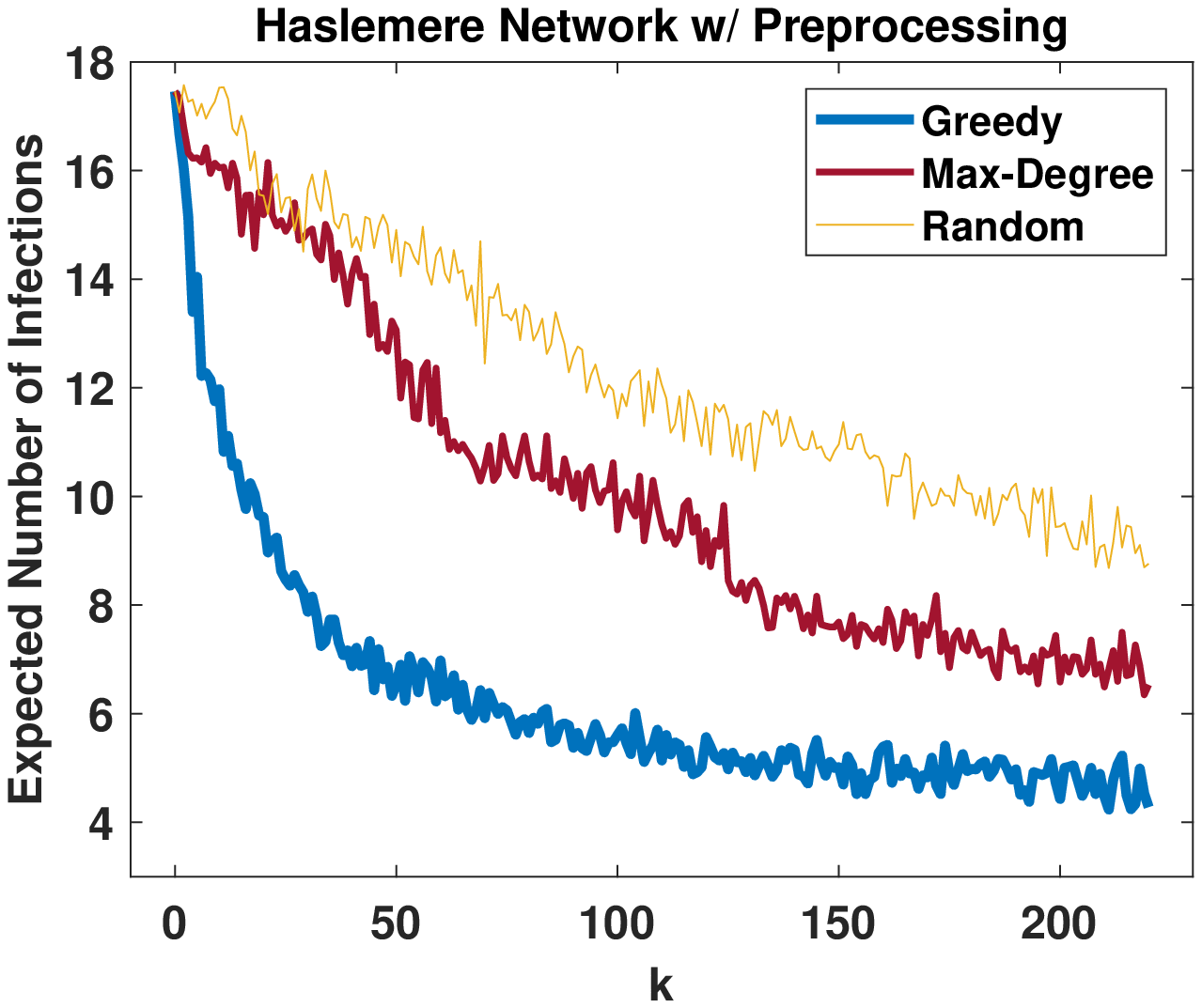}\quad
\includegraphics[width=0.4\textwidth]{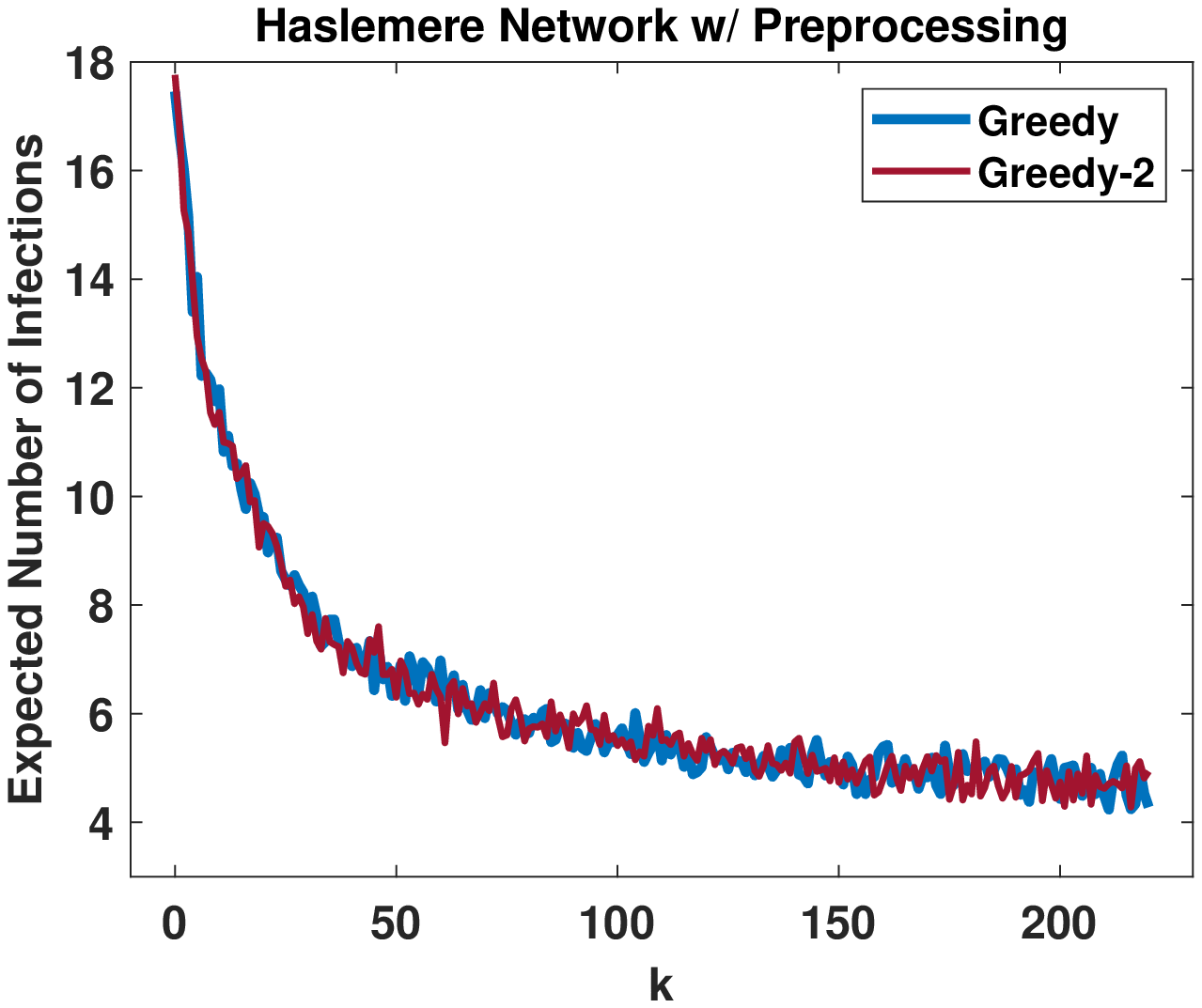}
\caption{In the IC-SIR model, the expected number of infections $\expec{}{\tilde{\sigma}}$ for the edge deletion sets given by Greedy, Max-Degree, and Random algorithms on the ER network, the contact network generated from the SBM model, the Haslemere contact network with preprocessing.
The bottom-right figure compares the results of Algorithm~\ref{alg:greedy} with the input $\expec{}{\tilde{\sigma}'}$ (Greedy) and $\expec{}{\tilde{\sigma}}$ (Greedy-2). 
The title of the horizontal axis $k$ is the number of removed edges.}
\label{fig:IC_SIR_RHO}
\end{figure}

We first run experiments on an ER network $G(V,E)$ with $1251$ edges sampled from $\calG(n,p_1)$ where $n=500$ and $p_1=0.01$. The activation probability $p$ for each edge in $E$ is $0.16$. The candidate set with $625$ edges is chosen from the contact network uniformly at random. We 
randomly and uniformly 
choose $5$ nodes in $V$ as seeds. 

Second, we consider a SBM network sampled from $\SBM(n,\kappa,\QQ)$, where $n=100$, $\kappa=5$, the entries $Q_{ij}$ for $i,j\in[\kappa]$ are set as $Q_{ij}=0.023$ for $i=j$ and $Q_{ij}\in [0.0036, 0.0046]$, chosen uniformly at random. The generated contact network has $995$ edges. The candidate set with $497$ edges is chosen from the contact network uniformly at random. The activation probability $p$ is set as $0.21$. The set of seeds with $5$ nodes is chosen from the network uniformly at random.

Third, we run experiments on the Haslemere network with preprocessing, introduced in Section \ref{subsec:dsir_sim}. For this network, we let the connection probability be $0.179$ and uniformly sample half of the edges in the contact networks as the candidate sets. 

The results show that our algorithm outperforms the two considered heuristics in minimizing the expected number of infections in all 
three networks. Since we have proved that the expected number of infections $\expec{}{\tilde{\sigma}}$ is a $(1\pm o(1))$-approximation of  $\expec{}{\tilde{\sigma}'}$, we also apply Algorithm~\ref{alg:greedy} with the input $\expec{}{\tilde{\sigma}}$ to the preprocessed Haslemere network. In Figure~\ref{fig:IC_SIR_RHO}, we compare the results returned by  Algorithm~\ref{alg:greedy} with the input $\expec{}{\tilde{\sigma}'}$ (Greedy) and $\expec{}{\tilde{\sigma}}$ (Greedy-2), respectively. The results show that the performance of Greedy and Greedy-2 are almost identical.

Then we show the performance of Greedy-2 on the Haslemere network with and without preprocessing in Figure~\ref{fig:IC_SIR}. In the Haslemere network without preprocessing, the Greedy-2 algorithm is marginally better than the two heuristics. 
Similar degeneration of the
effectiveness of the greedy algorithm is 
observed in networks with power-law degree distributions.
These examples highlight that our algorithm does not perform well on networks with relatively high max degree, as expected from the discussion in Section~\ref{dis:robustness}.

\begin{figure}
\centering
\includegraphics[width=0.4\textwidth]{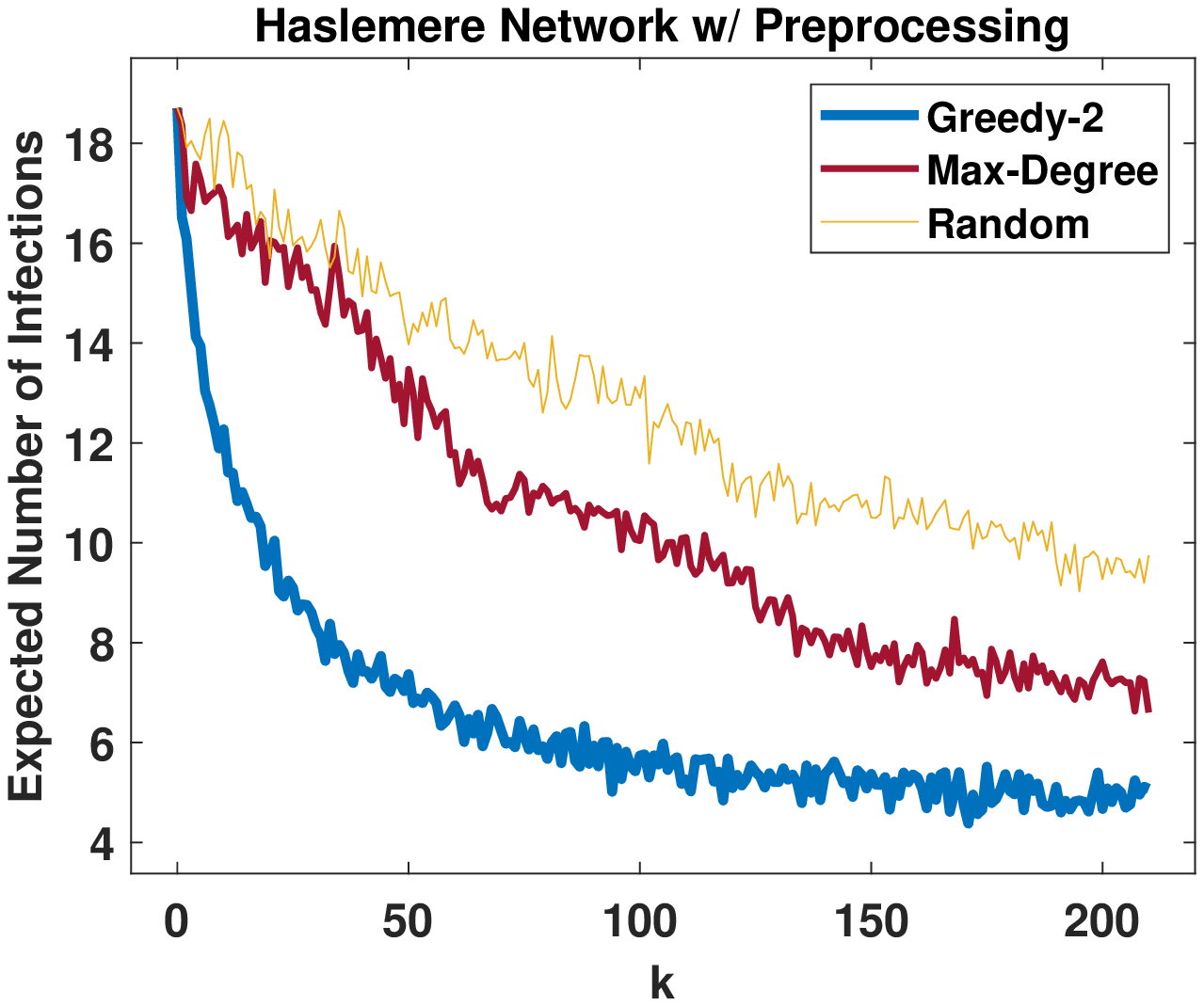}\quad
\includegraphics[width=0.4\textwidth]{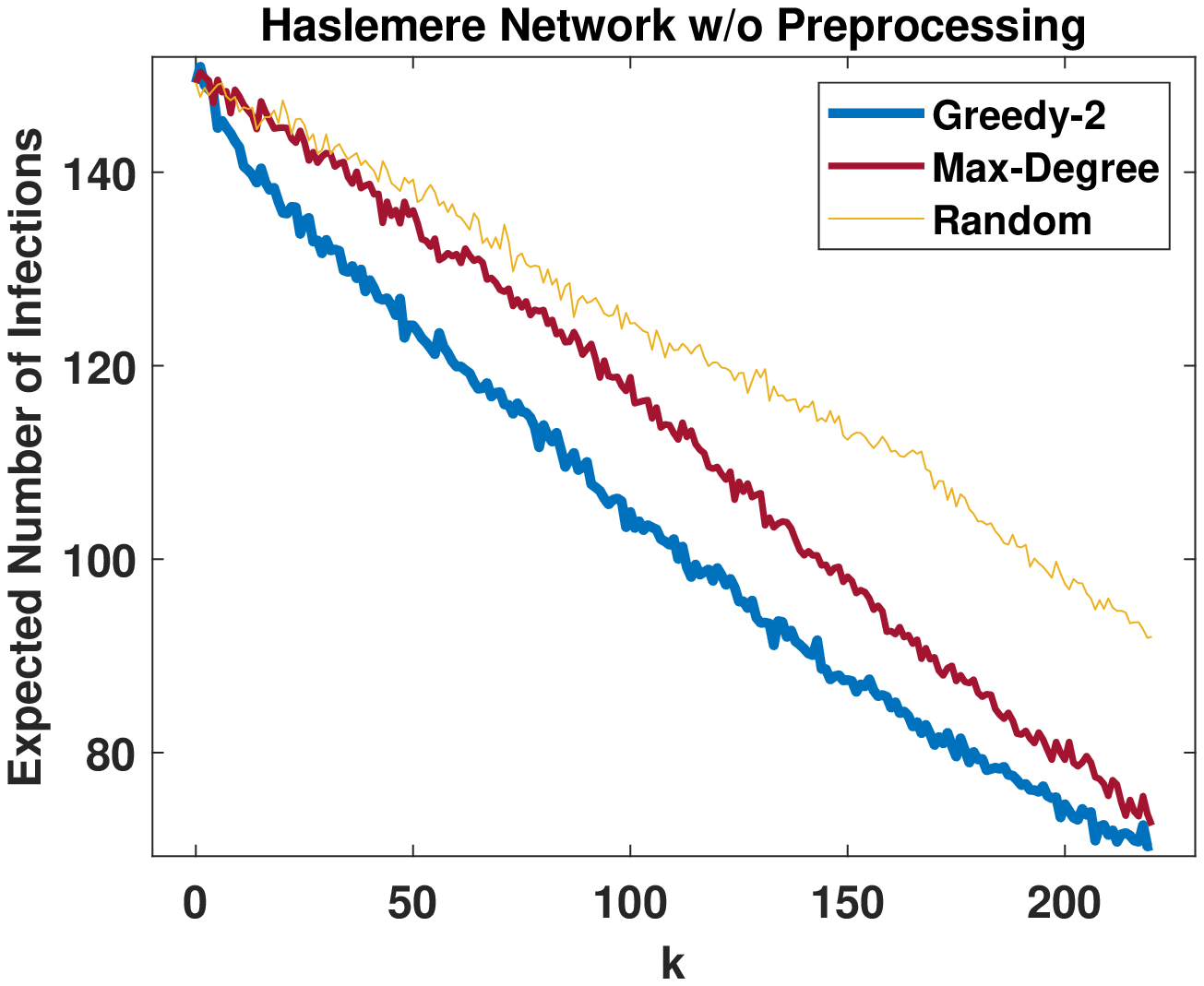}
\caption{In the IC-SIR model, the expected number of infections $\expec{}{\tilde{\sigma}}$ for the edge deletion sets given by Greedy-2, Max-Degree, and Random algorithms on the Haslemere contact network with and without preprocessing.
The title of the horizontal axis $k$ is the number of removed edges.}
\label{fig:IC_SIR}
\end{figure}

\subsection{Comparison of Two Proposed Algorithms}

Here we compare the performance of the greedy algorithm in Algorithm \ref{alg:greedy} when using objectives $\hat{\sigma}$ from the D-SIR model versus $\expec{}{\tilde{\sigma}}$ from the IC-SIR model. We implement the results of the algorithms on the G-SIR model to evaluate their effectiveness, while also considering computation costs. 

For both models we run the 
greedy algorithm on an ER network with $50$ nodes and connection probability $p=0.08$. The generated network has $98$ edges.  We choose $5$ fixed nodes as the set of seeds. We run the G-SIR simulations $15000$ times. According to the Hoeffding's bound, we have a probability of
at least $0.9$ to obtain an error less than or equal to $0.25$. For the D-SIR model, 
Figure~\ref{fig:Accuracy} 
shows the expected number of infections in the G-SIR model, the upper bound $\hat{\sigma}$, and the number of
infections $\sigma$ in the D-SIR model. In the IC-SIR model, the activation probability of an edge $(j,i)$ is set to be the cumulative infection rate from node $j$ to $i$, which is $\sum_{t=0}^{\infty}((1-D_{i})^t(1-B_{ij})^t B_{ij})=(1-(1-D_{i})(1-B_{ij}))^{-1} B_{ij}$. The figure shows the expected number of infections estimated in the G-SIR model and the corresponding IC-SIR model. We show that the two algorithms perform similarly in minimizing the number of infections. The IC-SIR model gives a better approximation for the G-SIR model in terms of the expected number of infections. 

In our implementations, the running time of the algorithm for the D-SIR model is $O(n^3+k(n^2+qn))$; the running time of the algorithm for the IC-SIR model with the input $\expec{}{\tilde{\sigma}'}$ or $\expec{}{\tilde{\sigma}}$ is 
$O(kqn^2\ln{n}/\eps^2)$.
Since we let $q=O(n)$ and $k=O(q)$ in all the experiments, the algorithm for the D-SIR model runs  much faster than the algorithm for the IC-SIR model. We note that the running time of the latter algorithm can be improved by designing data structures to efficiently maintain the connectivity of the generated contagion networks upon edge deletion. 
Discussion about this topic is beyond the scope of this paper.

\begin{figure}
\centering
\includegraphics[width=0.4\textwidth]{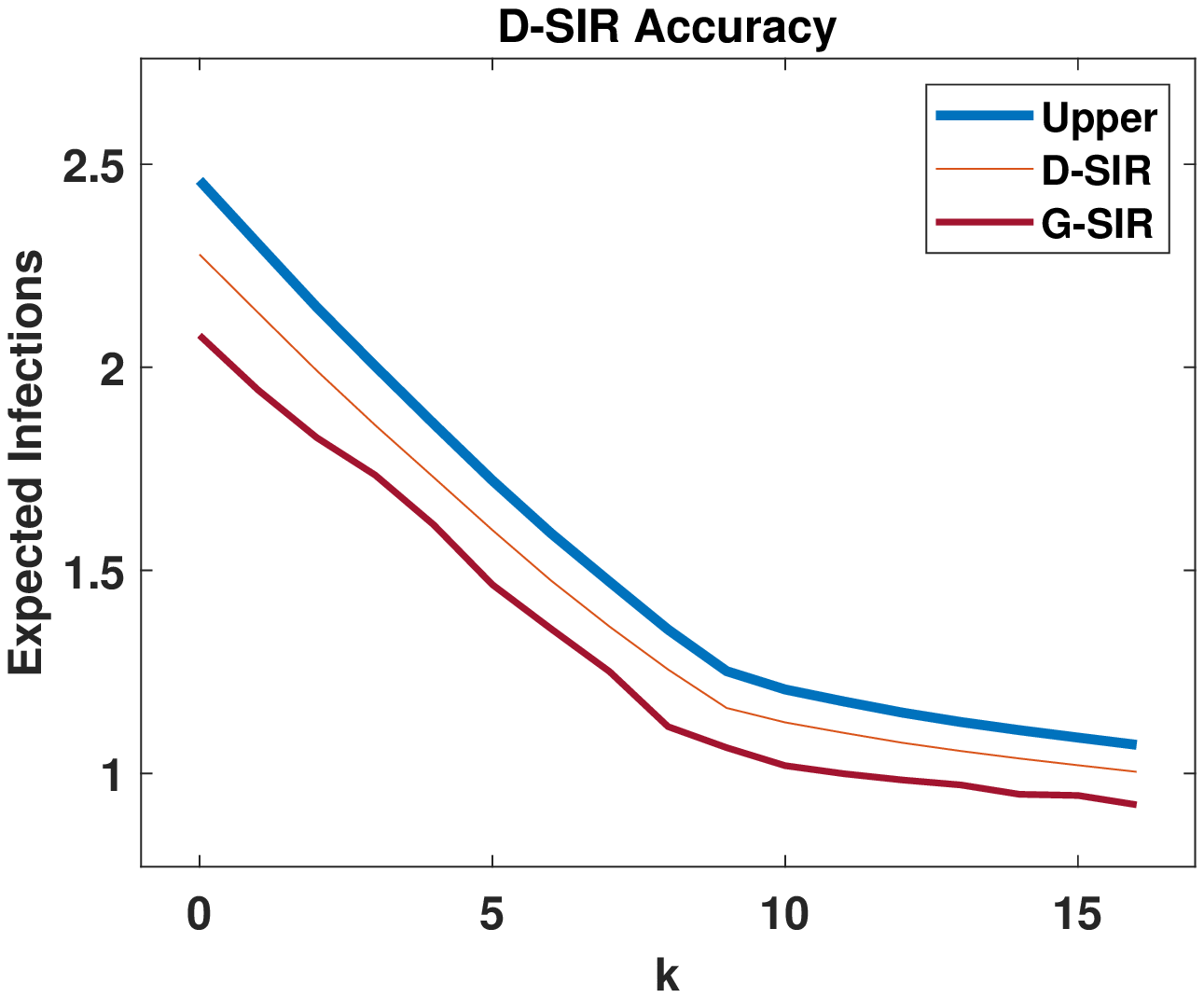}\quad
\includegraphics[width=0.4\textwidth]{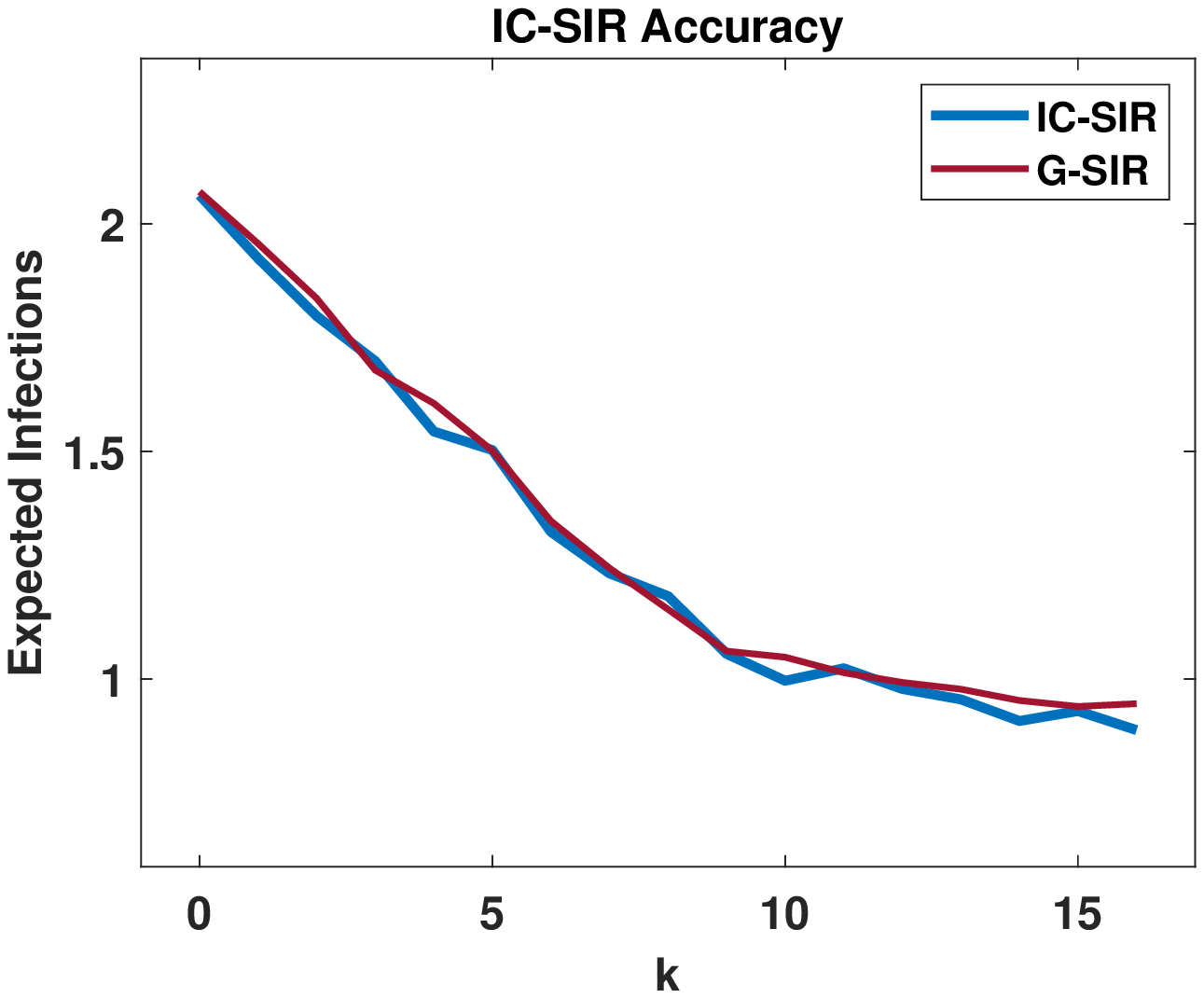}
\caption{The expected number of infections for the edge deletion sets returned by the greedy algorithm in Algorithm \ref{alg:greedy} using objectives $\hat{\sigma}$ in the D-SIR model and $\expec{}{\tilde{\sigma}'}$ in the IC-SIR model on the ER network. We compare the results  
of the algorithm by implementing them on
the G-SIR model with the approximations $\hat{\sigma}$ (Upper), $\sigma$ (D-SIR), and the estimated  $\expec{}{\tilde{\sigma}}$ (IC-SIR).}
\label{fig:Accuracy}
\end{figure}
\section{Conclusion}

In this paper we studied the expected number of infections in a Markov chain SIR model and its two instantiations: the D-SIR model and the IC-SIR model. We focused on the problem of minimizing the number of infections by deleting edges from a given candidate edge deletion set in these models. We proposed a greedy algorithm for each model, and provided guarantees for the performance of these algorithms under certain conditions. 
For the D-SIR model, we proposed a monotone supermodular upper bound for the number of infections when the system is exponentially stable.  For the IC-SIR model we proved that the expected number of infections is a close approximation to a monotone supermodular function with high probability when 1) the contact network is sampled from an ER graph or a SBM with a small number of blocks; 2) the number of initial infections are relatively small; 3) the average degree in the contagion network is less than $1$.
We validated the accuracy of the two models and the effectiveness of the proposed algorithms by running experiments on several synthesized networks and a real contact network. The results showed that the two models 
approximate the Markov chain SIR model well. The
simulations show that the
proposed algorithms can be used to minimize the number of infections effectively when the network does not contain highly connected nodes. 
Therefore, from this we learn that
in order to enable effective targeted mitigation strategies based on contact tracing, the maximum degree in the network needs to be contained by implementing  general measures.

Future work can consider the edge removing problem in a time-varying contact network. The problem can be formulated as an offline or an online problem. In addition, the contact network is assumed to be known without uncertainty in this paper. Designing adaptive algorithms to update the edge deletion set using new observations of the network or the process is another interesting future direction.

\appendix

\section{
Proof from Section~\ref{sec:dsir}}\label{apx:dsir}

\begin{proof}[Proof of Lemma~\ref{lem:dsir_supermodular}]
Since $\norm{\MM_{-P}}<1$ for any set $P \subseteq Q$, the inverse of matrix $(\II-\MM_{-P})$ always exists.
We note that $(\II-\MM_{-P})$ is an M-matrix for any edge deletion set $P$, which implies that its inverse $(\II-\MM_{-P})^{-1}$ is entry-wise non-negative by Lemma~\ref{lem:m-matrix}. 

Consider adding an edge $(j,i)$ to the deletion
set $P$. 
By the Sherman-Morrison formula~\cite{Bar51}, we have
\begin{align*}
    &(\II-\MM_{-P} + (1-x_i(0)-r_i(0))B_{ij} \ee_i \ee_{j}^{\top} )^{-1} \\
    =& (\II-\MM_{-P})^{-1} - c\cdot
    \frac{(\II-\MM_{-P})^{-1}\ee_i \ee_j^{\top}(\II-\MM_{-P})^{-1}}
    {1 + c \ee_j^\top (\II-\MM_{-P})^{-1} \ee_i}\,,
\end{align*}
in which $c = (1-x_i(0)-r_i(0))B_{ij}$ is a non-negative number.
Let $H(P)\defeq (\II-\MM_{-P})^{-1}$.
Thus, we have 
that
$H(P)$ is monotonically
nonincreasing as a function of  adding edges to the deletion set.

For an edge  deletion
set $P$ and any edge $e=(j,i) \in Q \setminus P$, we derive
\begin{align*}
    H(P) - H(P \cup (j,i)) = \int_{0}^{1} c \cdot (\II-\MM_{-P})^{-1} \ee_i \ee_j^{\top}(\II-\MM_{-P})^{-1} t \, \mathrm{d}t\,.
\end{align*}

Considering any set $P_2$ satisfying $P_1\subseteq P_2$ and $e\notin P_2$, we have the same result except that $\MM_{-P}$ is replaced by $\MM_{-P_2}$. 
Since we
proved that the function $H(S)$ is
monotonically
nonincreasing, 
we have that
\begin{align*}
    H(P_1) - H(P_1 \cup (j,i)) \geq H(P_2) - H(P_2 \cup (j,i))\,.
\end{align*}
Therefore, we proved that function $H(P)$ is entry-wise non-negative, and 
an entry-wise monotone supermodular 
function
with respect to the set of edges removed. 

In addition, we observe that $(\MM_{-P}+\DDelta-\II)$ is an entry-wise nonnegative nonincreasing modular function of $P$. Therefore, the product of $(\MM_{-P}+\DDelta-\II)$ and $H(P)$ is entry-wise monotone supermodular. Since the entries of $\1$ and $\xx(0)$ are all non-negative, the function $\hat{\sigma}(P)$ is monotone supermodular.
\end{proof}


\section{
Proofs from  Section~\ref{sec:IC_R}}\label{apx:icsir_random}

\begin{proof}[Proof of Lemma~\ref{lem:cc_size}]
We couple the random graph model with the branching process~\cite[Section 11.3]{AS15}, which starts from a root and samples $n$ $i.i.d.$ Bernoulli random variables with probability $p$ for an active leaf to generate its children at each time step. Let $Y_t$ be the number of active leaves at time step $t$. If $Y_{t-1} > 0$, then let $Z_t$ be the number of new leaves generated at time $t$. Thus, each random variable $Z_t$ is
the sum of $n$ $i.i.d.$ Bernoulli random variables with probability $p$.  Then, we consider the recursion $Y_0 = 1$ and $Y_t = Y_{t-1} + Z_t -1$, for all $t\geq 1$. The number of vertices generated by this process is an upper bound of the size of the corresponding connected component in the random graph model. 

Let $T$ be the stopping time of this branching process, which is also the number of vertices generated in this process. Then, we have that
the probability of generating more than $t$ vertices is
\begin{align*}
\prob{}{T > t} \leq \prob{}{Y_t > 0} = \prob{}{\sum_{i=1}^t Z_t > t}.
\end{align*}
Since $\{Z_i\}_{i=1}^t$ are independent and each $Z_i$ are the sum of $n$ $i.i.d.$ Bernoulli random variables, by the Chernoff bound, we have
\begin{align*}
\prob{}{\sum_{i=1}^t Z_t > t} \leq \exp\left(-\frac{(1-d)^2 t}{3}\right).
\end{align*}
By setting $t = L$, we get $\prob{}{T > L}  \leq n^{-3}$. Using the union bound, we get the probability of any node in a connected component greater than $L$ is at most $n^{-2}$. 
\end{proof}


\begin{proof}[Proof of Lemma~\ref{lem:seed_collision}]
Given that the event $\bar{\tau}_1$ happens, we can put all connected components in bins with the same capacity $L$. Let $m$ be the number of bins we used. All connected components can be placed into less than $\lceil\frac{2n}{L}\rceil$ bins because we can place the connected components in a way that each bin is at least half full, with possibly one extra bin.
 
For any fixed bin, the probability that a seed is sampled from any components from this bin is at most $L/n$.
By using the union bound, the probability of having at least one bin with more than one seed is 
\begin{align*}
    \prob{}{\tau_2 \mid \bar{\tau}_1}  \leq \binom{s}{2} m
    \left(\frac{L}{n}\right)^2
    \leq \frac{2s^2L}{n}. 
\end{align*}
\end{proof}


\begin{proof}[Proof of Lemma~\ref{lem:cycle}]
Since deleting edges might remove some cycles in the network, we can upper bound $y(P)$ by $y(\emptyset)$. We first estimate the number of cycles $x$ in the random graph $\calG(n,p)$. According to the analysis in~\cite[Theorem 8.9]{BHK20}, the number of possible length-$z$ cycles is $\binom{n}{z} \frac{(z-1)!}{2}$. The probability that such a length-$z$ cycle appears in the random graph is $p^z$. Therefore, the expected number of cycles in the random graph is at most
\begin{align}
\label{cycle:eqn}
    \expec{\widetilde{G}}{x} \leq \sum_{z=3}^{n} \binom{n}{z} \frac{(z-1)!}{2} p^z \leq \frac{1}{2}\sum_{z=3}^{n} d^z \leq \frac{d^3}{2(1-d)} \,,
\end{align}
where the second inequality follows from the fact that $p=d/n$.

Suppose the event $\bar{\tau}_1$ happens, then all connected components in the contagion network have size at most $L$.  Since all seeds are chosen uniformly at random, the probability that each seed is sampled from a connected component with cycles is at most $xL/n$. Thus, for any fixed random graph, the expected number of connected components with at least one seed and one cycle is at most $sxL/n$, where $s$ is the number of seeds. 
Since the size of each component is at most $L$, the expectation of $y(P)$ over the set of seeds $\mathcal{S}$ is
\begin{align*}
    \expec{\mathcal{S}}{y(P) \mid \widetilde{G},\bar{\tau}_1} \leq L \cdot \frac{sxL}{n}.
\end{align*}
Therefore, we have
\begin{align*}
    \expec{}{y(P) \mid \bar{\tau}_1,\bar{\tau}_2}
    \prob{}{\bar{\tau}_1,\bar{\tau}_2}
    \leq& \expec{}{y(P) \mid \bar{\tau}_1}
    \prob{}{\bar{\tau}_1} \leq \expec{\widetilde{G}}{L\cdot \frac{sxL}{n} \mid \bar{\tau}_1}\prob{}{\bar{\tau}_1}\\
    =& \frac{sL^2}{n}\expec{\widetilde{G}}{x \mid \bar{\tau}_1}\prob{}{\bar{\tau}_1}
    \leq \frac{sL^2}{n}\frac{d^3}{2(1-d)},
\end{align*}
where the second inequality follows from
taking the expectation over seeds and the last inequality holds because $\expec{\widetilde{G}}{x \mid \bar{\tau}_1}\prob{}{\bar{\tau}_1} \leq \expec{\widetilde{G}}{x} \leq d^3/(2-2d)$.
\end{proof}

\section{
Proof from Section~\ref{sec:IC_SBM}}\label{apx:IC_SBM}

\begin{proof}[Proof of Lemma~\ref{lem:sbm_cycles}]
An instance of the stochastic block model $\SBM(n,\kappa,\QQ)$ can be viewed as a composition of $\kappa$ random graphs and $\kappa(\kappa-1)$ random bipartite graphs.
First, we generate all the intra-block edges using probabilities given by $\diag{\QQ}$. By Lemma~\ref{lem:cc_size} we obtain that each block contains a connected component of size greater than $L_{\rm{init}}=9(1-d_{\rm{init}})^{-2}\ln{n}$ with probability at most $n^{-2}$. By the union bound, the probability that there exists a connected component greater than $L_{\rm{init}}=9(1-d_{\rm{init}})^{-2}\ln{n}$ in the whole network is less or equal to $\kappa n^{-2}$. 

Similar to the analysis of the 
ER random graphs, we place the connected components into bins with the same capacity $L_{\rm{init}}$, and we refer to these bins as small bins hereafter. 
In addition, according to (\ref{cycle:eqn}), the expected number of cycles given that $\bar{\tau}_1$ happens satisfies
\begin{align*}
\expec{}{x \mid \bar{\tau}_1}\leq \kappa \cdot \frac{(d_{\rm{init}})^3}{2(1-d_{\rm{init}})}\cdot\frac{1}{\prob{}{\bar{\tau}_1}}\,.
\end{align*}

We denote the current network by $\widetilde{G}^{(\kappa)}$, indicating that $\kappa$ random graphs have been generated. Then we add the random bipartite graphs to the graph until we obtain $\widetilde{G}=\widetilde{G}^{(\kappa^2)}$. 

Suppose at some point we have constructed the graph $\widetilde{G}^{(\kappa+\alpha-1)}$, with $\kappa$ random graphs and $\alpha-1$ random bipartite graphs generated. 
We consider the expected incremental number of cycles added to the graph when we add the $\alpha$-th random bipartite graph $\widetilde{G}_{\rm{bip}}^{(\alpha)}$ with connecting probability $\QQ_{ij}$. The set of new cycles includes cycles that only use edges in $\widetilde{G}_{\rm{bip}}^{(\alpha)}$ and cycles formed by connecting existing connected components in $\widetilde{G}^{(\kappa+\alpha-1)}$.

The number of cycles using only edges in $\widetilde{G}_{\rm{bip}}^{(\alpha)}$ is given by
\begin{align*}
    \expec{}{x^{\rm{bip}}_{(\alpha+1)} \mid \bar{\tau}_1}\prob{}{\bar{\tau}_1} &\leq \sum_{z=2}^{\lceil L^*/2\rceil} \binom{n}{z}^2
    \frac{(z!)^2}{2}
    \left(\QQ_{ij}\right)^z \leq \frac{1}{2}\sum_{z=2}^{\lceil L^*/2\rceil} \left(n \QQ_{ij}\right)^z \leq \frac{(d_{ij})^2}{2(1-d_{ij})} \,.
\end{align*}
Then we consider the expected number of cycles formed by connecting exactly $z$ distinct existing connected components. For any $z$ connected components $\{C_i\}_{i=1}^z$ with sizes $\{\ell_i\}_{i=1}^z$, respectively, let $\calC(C_1,C_2,\cdots,C_z)$ be the event that they form a cycle. Then, we have
\begin{align*}
    &\prob{}{\calC(C_1,C_2,\cdots,C_z)}  \leq (z-1)!\cdot\frac{\ell_z\ell_1}{n}\prod_{i=1}^{z-1} \frac{\ell_i\ell_{i+1}}{n} \\
    &\leq \frac{(z-1)!}{n^z}\prod_{i=1}^z \ell_i^2 \leq \frac{(z-1)!}{n^z} \left(\frac{\sum_{i=1}^z \ell_i}{z}\right)^{2z} \leq \frac{(z-1)!}{n^z} \left(\frac{L^*}{z}\right)^{2z},
\end{align*}
where the last inequality is due to the generalized mean inequality. Note that this upper bound does not depend on the sizes of components, which can be denoted by $\prob{}{\calC_z}$.

For any fixed $z$, we use $N_z$ to denote the number of choices for these $z$ components. Let $m_{\rm{bin}}$ be the number of small bins needed to cover a large bin, which is at most
\begin{align*}
m_{\rm{bin}} \leq 
\left\lceil\frac{L^*}{L_{\rm{init}}/2}\right\rceil
\leq 2\left\lceil\left(\frac{1-d_{\rm{init}}}{1-d_{\rm{end}}}\right)^2\frac{\ln(\kappa n)}{\ln n}\right\rceil \leq 4\left\lceil\left(\frac{1-d_{\rm{init}}}{1-d_{\rm{end}}}\right)^2\right\rceil.
\end{align*}
Then, we consider two cases regrading to the number of components $z$.

If the number of components $1 \leq z \leq m_{\rm{bin}}$, the number of choices $N_z$ satisfies
\begin{align*}
N_z \leq \binom{n\kappa}{z} \binom{L^*}{z},
\end{align*}
where the first factor upper bounds the number of choices of $z$ small bins in the graph, and the second factor upper bounds the number of choices of $z$ connected components from the $z$ bins, which are all in a new connected component whose size is at most $L^*$.
Then the expected number of cycles for this case is 
\begin{align*}
    &\expec{}{x'_{(\alpha)} \mid \bar{\tau}_1} \prob{}{\bar{\tau}_1} \leq \sum_{z=1}^{m_{\rm{bin}}} N_z \cdot \prob{}{\calC_z} \leq \sum_{z=1}^{m_{\rm{bin}}} \binom{n\kappa}{z} \binom{L^*}{z} \cdot \prob{}{\calC_z} \\
    \leq& \sum_{z=1}^{m_{\rm{bin}}} \frac{(n\kappa)^z}{z!} \cdot \frac{(L^*)^z}{z!} \cdot \frac{(z-1)!}{n^z}\left(\frac{L^*}{z}\right)^{2z} =  \sum_{z=1}^ {m_{\rm{bin}}} \frac{ (L^*)^{3z}\cdot \kappa^z}{z^{2z+1}\cdot z!} \leq \sum_{z=1}^ {m_{\rm{bin}}} (L^*)^{3m_{\rm{bin}}}\cdot \kappa^{m_{\rm{bin}}}.
\end{align*}

If the number of components $z> m_{\rm{bin}}$, then $N_z$ satisfies
\begin{align*}
N_z \leq \binom{n\kappa}{m_{\mathrm{bin}}} \binom{L^*}{z}.
\end{align*} 
The expected number of cycles in this case is
\begin{align*}
    &\expec{}{x''_{(\alpha)} \mid \bar{\tau}_1} \prob{}{\bar{\tau}_1} \leq \sum_{z=m_{\rm{bin}}+1}^{L^*} N_z \cdot \prob{}{\calC_z}
    \\
    \leq & \sum_{z=m_{\rm{bin}}+1}^{L^*} \frac{(n\kappa)^{m_{\rm{bin}}}}{m_{\rm{bin}}!}\cdot \frac{(L^*)^z}{z!} \cdot \frac{(z-1)!}{n^z} \left(\frac{L^*}{z}\right)^{2z}\\
    = & \sum_{z=m_{\rm{bin}}+1}^{L^*} \frac{(L^*)^{3z}\cdot \kappa^{m_{\rm{bin}}}}{n^{z-m_{\rm{bin}}}\cdot z^{2z+1}\cdot m_{\rm{bin}}!} \leq \sum_{z=m_{\rm{bin}}+1}^{L^*} (L^*)^{3m_{\rm{bin}}}\cdot \kappa^{m_{\rm{bin}}}.
\end{align*}

Combining these two cases, we have the expected number of new cycles formed by connecting existing connected components is 
\begin{align*}
\expec{}{x^{\rm{con}}_{(\alpha)} \mid \bar{\tau}_1} \prob{}{\bar{\tau}_1} \leq \expec{}{x'_{(\alpha)} + x''_{(\alpha)} \mid \bar{\tau}_1} \prob{}{\bar{\tau}_1} \leq (L^*)^{3m_{\rm{bin}}+1} \cdot \kappa^{m_{\rm{bin}}}\,,
\end{align*}
and the total expected number of cycles that are created upon adding the random bipartite graph $\widetilde{G}_{\rm{bip}}^{(\alpha)}$ is
\begin{align*}
\expec{}{x_{(\alpha)} \mid \bar{\tau}_1} \prob{}{\bar{\tau}_1} = \expec{}{x^{\rm{bip}}_{(\alpha)} + x^{\rm{con}}_{(\alpha)}\mid \bar{\tau}_1} \prob{}{\bar{\tau}_1}
= \frac{(d_{ij})^2}{2(1-d_{ij})}+ (L^*)^{3m_{\rm{bin}}+1} \cdot \kappa^{m_{\rm{bin}}}. 
\end{align*}
By adding all random bipartite graphs, we obtain that the expected number of cycles in $\widetilde{G}$ is
\begin{align*}
\expec{}{x \mid \bar{\tau}_1}\prob{}{\bar{\tau}_1} &= \kappa\cdot \frac{(d_{\rm{init}})^3}{2(1-d_{\rm{init}})} + \kappa(\kappa-1)(L^*)^{3m_{\rm{bin}}+1} \cdot \kappa^{m_{\rm{bin}}} + \sum_{\substack{i,j\in [\kappa]\\i\neq j}}\frac{(d_{ij})^2}{2(1-d_{ij})}\\
&\leq \kappa\cdot \frac{(d_{\rm{init}})^3}{2(1-d_{\rm{init}})} + (L^*)^{3m_{\rm{bin}}+1} \cdot \kappa^{m_{\rm{bin}}+2} + \kappa^2\frac{(d_{\rm{end}})^2}{2(1-d_{\rm{end}})}.
\end{align*}
For any $\kappa=O(\ln{n})$, we obtain
\begin{align*}
\expec{}{y \mid \bar{\tau}_1,\bar{\tau}_2}\prob{}{\bar{\tau}_1,\bar{\tau}_2}\leq& \expec{}{y \mid \bar{\tau}_1}\prob{}{\bar{\tau}_1} \\
\leq& \expec{}{L^* \frac{sxL^*}{\kappa n} \mid \bar{\tau}_1}\prob{}{\bar{\tau}_1} = O\left(\frac{s(\ln n)^{4m_{\rm{bin}}}}{n}\right).\tag*{\qedhere}
\end{align*}
\end{proof}

\newpage
\bibliographystyle{siamplain}
\bibliography{ref}

\begin{thebibliography}{10}

\bibitem{abbey1952examination}
{\sc H.~Abbey}, {\em An examination of the {R}eed-{F}rost theory of epidemics},
  Human Biology, 24 (1952), p.~201.

\bibitem{HH13}
{\sc H.~J. Ahn and B.~Hassibi}, {\em Global dynamics of epidemic spread over
  complex networks}, in Proceedings of the 52nd IEEE Conference on Decision and
  Control (CDC), IEEE, 2013, pp.~4579--4585.

\bibitem{albert2000error}
{\sc R.~Albert, H.~Jeong, and A.-L. Barab{\'a}si}, {\em Error and attack
  tolerance of complex networks}, {N}ature, 406 (2000), pp.~378--382.

\bibitem{AS15}
{\sc N.~Alon and J.~H. Spencer}, {\em The Probabilistic Method}, John Wiley \&
  Sons, Hoboken, New Jersy, 4~ed., 2015.

\bibitem{Bar51}
{\sc M.~S. Bartlett}, {\em An inverse matrix adjustment arising in discriminant
  analysis}, The Annals of Mathematical Statistics, 22 (1951), pp.~107--111.

\bibitem{Bha18}
{\sc P.~Bhardwaj}, {\em Disrupting diffusion: Critical nodes in network},
  master's thesis, Iowa State University, 2018.

\bibitem{bishop2011link}
{\sc A.~N. Bishop and I.~Shames}, {\em Link operations for slowing the spread
  of disease in complex networks}, EPL (Europhysics Letters), 95 (2011),
  p.~18005.

\bibitem{BHK20}
{\sc A.~Blum, J.~Hopcroft, and R.~Kannan}, {\em Foundations of Data Science},
  Cambridge University Press, 2020.

\bibitem{BBCL14}
{\sc C.~Borgs, M.~Brautbar, J.~T. Chayes, and B.~Lucier}, {\em Maximizing
  social influence in nearly optimal time}, in Proceedings of the Twenty-Fifth
  Annual {ACM-SIAM} Symposium on Discrete Algorithms {(SODA)}, 2014,
  pp.~946--957.

\bibitem{borgs2010distribute}
{\sc C.~Borgs, J.~Chayes, A.~Ganesh, and A.~Saberi}, {\em How to distribute
  antidote to control epidemics}, Random Structures \& Algorithms, 37 (2010),
  pp.~204--222.

\bibitem{Bui86}
{\sc T.~N. Bui}, {\em Graph {B}isection {A}lgorithms}, PhD thesis,
  Massachusetts Institute of Technology, 1986.

\bibitem{callaway2000network}
{\sc D.~S. Callaway, M.~E. Newman, S.~H. Strogatz, and D.~J. Watts}, {\em
  Network robustness and fragility: Percolation on random graphs}, Physical
  {R}eview {L}etters, 85 (2000), p.~5468.

\bibitem{cator2014nodal}
{\sc E.~Cator and P.~Van~Mieghem}, {\em Nodal infection in {M}arkovian
  susceptible-infected-susceptible and susceptible-infected-removed epidemics
  on networks are non-negatively correlated}, Physical Review E, 89 (2014),
  p.~052802.

\bibitem{draief2008thresholds}
{\sc M.~Draief, A.~Ganesh, L.~Massouli{\'e}, et~al.}, {\em Thresholds for virus
  spread on networks}, The Annals of Applied Probability, 18 (2008),
  pp.~359--378.

\bibitem{enns2012optimal}
{\sc E.~A. Enns, J.~J. Mounzer, and M.~L. Brandeau}, {\em Optimal link removal
  for epidemic mitigation: A two-way partitioning approach}, Mathematical
  Biosciences, 235 (2012), pp.~138--147.

\bibitem{California_order}
{\sc {Executive Department, State of California}}, {\em {Executive Order
  N-33-20}}.
\newblock \url{https://covid19.ca.gov/img/Executive-Order-N-33-20.pdf}.
\newblock Online; accessed 26 October 2020.

\bibitem{firth2020using}
{\sc J.~A. Firth, J.~Hellewell, P.~Klepac, S.~Kissler, A.~J. Kucharski, and
  L.~G. Spurgin}, {\em Using a real-world network to model localized {COVID-19}
  control strategies}, Nature Medicine,  (2020), pp.~1--7.

\bibitem{GLM01}
{\sc J.~Goldenberg, B.~Libai, and E.~Muller}, {\em Talk of the network: A
  complex systems look at the underlying process of word-of-mouth}, Marketing
  Letters, 12 (2001), pp.~211--223.

\bibitem{GPSJ20}
{\sc S.~{Gracy}, P.~E. {Par\'e}, H.~{Sandberg}, and K.~H. {Johansson}}, {\em
  Analysis and distributed control of periodic epidemic processes}, IEEE
  Transactions on Control of Network Systems,  (2020), pp.~1--1.

\bibitem{HK02}
{\sc P.~Holme, B.~J. Kim, C.~N. Yoon, and S.~K. Han}, {\em Attack vulnerability
  of complex networks}, Physical Review E, 65 (2002), p.~056109.

\bibitem{hota2020closed}
{\sc A.~R. Hota, J.~Godbole, P.~Bhariya, and P.~E. Par{\'e}}, {\em A
  closed-loop framework for inference, prediction and control of {SIR}
  epidemics on networks}, arXiv preprint arXiv:2006.16185,  (2020).

\bibitem{KKT03}
{\sc D.~Kempe, J.~M. Kleinberg, and {\'{E}}.~Tardos}, {\em Maximizing the
  spread of influence through a social network}, in Proceedings of the Ninth
  {ACM} {SIGKDD} International Conference on Knowledge Discovery and Data
  Mining (KDD), 2003, pp.~137--146.

\bibitem{kermack1927contribution}
{\sc W.~O. Kermack and A.~G. McKendrick}, {\em A contribution to the
  mathematical theory of epidemics}, Proceedings of the royal society of
  london. Series A, 115 (1927), pp.~700--721.

\bibitem{khanafer2014optimal}
{\sc A.~Khanafer and T.~Basar}, {\em An optimal control problem over infected
  networks}, in Proceedings of the International Conference of Control, Dynamic
  Systems, and Robotics, 2014, pp.~1--6.

\bibitem{klepac2018contagion}
{\sc P.~Klepac, S.~Kissler, and J.~Gog}, {\em Contagion! the {BBC} four
  pandemic--the model behind the documentary}, Epidemics, 24 (2018),
  pp.~49--59.

\bibitem{maccluer2000many}
{\sc C.~R. MacCluer}, {\em The many proofs and applications of {Perron}'s
  theorem}, Siam Review, 42 (2000), pp.~487--498.

\bibitem{MA19}
{\sc V.~S. Mai and E.~H. Abed}, {\em Optimizing leader influence in networks
  through selection of direct followers}, {IEEE} Transactions on Automatic
  Control, 64 (2019), pp.~1280--1287.

\bibitem{mai2018distributed}
{\sc V.~S. Mai, A.~Battou, and K.~Mills}, {\em Distributed algorithm for
  suppressing epidemic spread in networks}, IEEE Control Systems Letters, 2
  (2018), pp.~555--560.

\bibitem{Sweden_order}
{\sc {Ministry of Justice, Sweden}}, {\em {Ordinance on a prohibition against
  holding public gatherings and events}}.
\newblock
  \url{https://www.government.se/articles/2020/03/ordinance-on-a-prohibition-against-holding-public-gatherings-and-events/}.
\newblock Online; accessed 26 October 2020.

\bibitem{Nem78}
{\sc G.~L. Nemhauser, L.~A. Wolsey, and M.~L. Fisher}, {\em An analysis of
  approximations for maximizing submodular set functions—{I}}, Mathematical
  Programming, 14 (1978), pp.~265--294.

\bibitem{NS12}
{\sc P.~Netrapalli and S.~Sanghavi}, {\em Learning the graph of epidemic
  cascades}, ACM Sigmetrics Performance Evaluation Review, 40 (2012),
  pp.~211--222.

\bibitem{OP16}
{\sc M.~{Ogura} and V.~M. {Preciado}}, {\em Efficient containment of exact
  {SIR} {M}arkovian processes on networks}, in Proceedings of the IEEE 55th
  Conference on Decision and Control (CDC), 2016, pp.~967--972.

\bibitem{pare2017epidemic}
{\sc P.~E. Par\'{e}, C.~L. Beck, and A.~Nedi\'{c}}, {\em Epidemic processes
  over time--varying networks}, IEEE Transactions on Control of Network
  Systems, 5 (2018), pp.~1322--1334.

\bibitem{pareautomatica}
{\sc P.~E. Par{\'e}, J.~Liu, C.~Beck, A.~Nedi\'c, and T.~Ba\c{s}ar}, {\em
  Multi-competitive viruses over time–varying networks with mutations and
  human awareness}, Automatica,  (2020).
\newblock {N}ote: {Accepted}.

\bibitem{plemmons1977m}
{\sc R.~J. Plemmons}, {\em M-matrix characterizations. {I}—nonsingular
  m-matrices}, Linear Algebra and its Applications, 18 (1977), pp.~175--188.

\bibitem{prakash2013fractional}
{\sc B.~A. Prakash, L.~Adamic, T.~Iwashyna, H.~Tong, and C.~Faloutsos}, {\em
  Fractional immunization in networks}, in Proceedings of the SIAM
  International Conference on Data Mining (ICDM), SIAM, 2013, pp.~659--667.

\bibitem{preciado2013optimal}
{\sc V.~M. Preciado, M.~Zargham, C.~Enyioha, A.~Jadbabaie, and G.~J. Pappas},
  {\em Optimal vaccine allocation to control epidemic outbreaks in arbitrary
  networks}, in Proceedings of the 52nd IEEE Conference on Decision and Control
  (CDC), IEEE, 2013, pp.~7486--7491.

\bibitem{PZEJ+14}
{\sc V.~M. Preciado, M.~Zargham, C.~Enyioha, A.~Jadbabaie, and G.~J. Pappas},
  {\em Optimal resource allocation for network protection against spreading
  processes}, IEEE Transactions on Control of Network Systems, 1 (2014),
  pp.~99--108.

\bibitem{RH15}
{\sc N.~A. {Ruhi} and B.~{Hassibi}}, {\em {SIRS} epidemics on complex networks:
  Concurrence of exact {M}arkov chain and approximated models}, in Proceedings
  of the 54th IEEE Conference on Decision and Control (CDC), 2015,
  pp.~2919--2926.

\bibitem{SMHH11}
{\sc C.~M. Schneider, T.~Mihaljev, S.~Havlin, and H.~J. Herrmann}, {\em
  Suppressing epidemics with a limited amount of immunization units}, Physical
  Review E, 84 (2011), p.~061911.

\bibitem{van2014exact}
{\sc P.~Van~Mieghem}, {\em Exact {Markovian} {SIR} and {SIS} epidemics on
  networks and an upper bound for the epidemic threshold}, arXiv preprint
  arXiv:1402.1731,  (2014).

\bibitem{van2009virus}
{\sc P.~Van~Mieghem, J.~Omic, and R.~Kooij}, {\em Virus spread in networks},
  IEEE/ACM Transactions on Networking, 17 (2009), pp.~1--14.

\bibitem{MSKL+11}
{\sc P.~Van~Mieghem, D.~Stevanovi{\'c}, F.~Kuipers, C.~Li, R.~Van De~Bovenkamp,
  D.~Liu, and H.~Wang}, {\em Decreasing the spectral radius of a graph by link
  removals}, Physical Review E, 84 (2011), p.~016101.

\bibitem{wan2007network}
{\sc Y.~Wan, S.~Roy, and A.~Saberi}, {\em Network design problems for
  controlling virus spread}, in Proceedings of the 46th IEEE Conference on
  Decision and Control (CDC), IEEE, 2007, pp.~3925--3932.

\bibitem{wan2008designing}
{\sc Y.~Wan, S.~Roy, and A.~Saberi}, {\em Designing spatially heterogeneous
  strategies for control of virus spread}, IET Systems Biology, 2 (2008),
  pp.~184--201.

\bibitem{WKSZ17}
{\sc X.~Wu, A.~Kumar, D.~Sheldon, and S.~Zilberstein}, {\em Robust optimization
  for tree-structured stochastic network design}, in Proceedings of the
  Thirty-First AAAI Conference on Artificial Intelligence, 2017,
  p.~4545–4551.

\bibitem{youssef2011individual}
{\sc M.~Youssef and C.~Scoglio}, {\em An individual-based approach to {SIR}
  epidemics in contact networks}, Journal of Theoretical Biology, 283 (2011),
  pp.~136--144.

\end{thebibliography}

\end{document}